\documentclass[12pt,a4paper,reqno]{amsart}
\usepackage{amsmath}
\usepackage{geometry}
\usepackage{amssymb, latexsym}
\usepackage{verbatim}
\usepackage{enumerate}
\usepackage{comment}
\usepackage{txfonts}

\usepackage{mathtools}
\usepackage[tableposition=top]{caption}
\usepackage{booktabs,dcolumn}

\newtheorem{theorem}{Theorem}[section]
\newtheorem{proposition}[theorem]{Proposition}

\newtheorem{lemma}[theorem]{Lemma}
\newtheorem{corollary}[theorem]{Corollary}

\newtheorem{definition}[theorem]{Definition}
\newtheorem{remark}[theorem]{Remark}

\newtheorem{example}[theorem]{Example}

\newcommand\E{\mathbb{E}}
\newcommand\R{\mathbb{R}}
\newcommand\Z{\mathbb{Z}}
\newcommand\N{\mathbb{N}}
\newcommand\D{\mathbb{D}}
\newcommand\C{\mathbb{C}}
\newcommand\Q{\mathbb{Q}}

\newcommand\eps{\varepsilon}
\newcommand\plim{\mathop{\widetilde \lim}}

\renewcommand\P{\mathbf{P}}

\begin{document}
\title[Correlations of multiplicative functions]{The structure of logarithmically averaged correlations of multiplicative functions, with applications to the Chowla and Elliott conjectures}

\author{Terence Tao}
\address{Department of Mathematics, UCLA\\
405 Hilgard Ave\\
Los Angeles CA 90095\\
USA}
\email{tao@math.ucla.edu}

\author{Joni Ter\"av\"ainen}
\address{Department of Mathematics and Statistics, University of Turku\\
20014 Turku\\
Finland}
\email{joni.p.teravainen@utu.fi}

\begin{abstract}	Let $g_0,\dots,g_k: \N \to \D$ be $1$-bounded multiplicative functions, and let $h_0,\dots,h_k \in \Z$ be shifts.  We consider correlation sequences $f: \N \to \Z$ of the form
$$ f(a) \coloneqq \plim_{m \to \infty} \frac{1}{\log \omega_m} \sum_{x_m/\omega_m \leq n \leq x_m} \frac{g_0(n+ah_0) \dots g_k(n+ah_k)}{n} $$
where $1 \leq \omega_m \leq x_m$ are numbers going to infinity as $m \to \infty$, and $\plim$ is a generalised limit functional extending the usual limit functional.  We show a structural theorem for these sequences, namely that these sequences $f$ are the uniform limit of periodic sequences $f_i$.  Furthermore, if the multiplicative function $g_0 \dots g_k$ ``weakly pretends'' to be a Dirichlet character $\chi$, the periodic functions $f_i$ can be chosen to be $\chi$-isotypic in the sense that $f_i(ab) = f_i(a) \chi(b)$ whenever $b$ is coprime to the periods of $f_i$ and $\chi$, while if $g_0 \dots g_k$ does not weakly pretend to be any Dirichlet character, then $f$ must vanish identically.  As a consequence, we obtain several new cases of the logarithmically averaged Elliott conjecture, including the logarithmically averaged Chowla conjecture for odd order correlations.  We give a number of applications of these special cases, including the conjectured logarithmic density of all sign patterns of the Liouville function of length up to three, and of the M\"obius function of length up to four.
\end{abstract}

\maketitle

\section{Introduction}

This paper is concerned with the structure of multiple correlations between bounded multiplicative functions.  Before we present our results in full generality, however, we first focus on the special case of correlations of the Liouville function, which have attracted particular attention in the literature.

\subsection{The Liouville function and the Chowla conjecture}

Let $\lambda: \N \to \{-1,+1\}$ denote the Liouville\footnote{For definitions of all the standard arithmetic functions used in this paper, see Section \ref{notation-sec}.} function, thus $\lambda(n)$ is equal to $+1$ when $n$ is the product of an even number of primes, and $-1$ otherwise. The \emph{Chowla conjecture} \cite{chowla} asserts that for any $k \geq 0$ and any distinct integers $h_0,\dots,h_k$, we have
\begin{equation}\label{limx-chowla}
 \lim_{x \to \infty} \frac{1}{x} \sum_{n \leq x} \lambda(n+h_0) \dots \lambda(n+h_k) = 0,
\end{equation}
where we adopt the convention that $\lambda(n)=0$ when $n \leq 0$.  Using the averaging notation
\begin{equation}\label{avg-def}
 \E_{n \in A} f(n) \coloneqq \frac{1}{|A|} \sum_{n \in A} f(n),
\end{equation}
where $|A|$ denotes the cardinality of a finite non-empty set $A$, we can also rewrite \eqref{limx-chowla} as
$$ \lim_{x \to \infty} \E_{n \leq x} \lambda(n+h_0) \dots \lambda(n+h_k) = 0.$$
It is a well known observation of Landau that the $k=0$ case of the Chowla conjecture is equivalent to the prime number theorem.  However, for $k \geq 1$ the conjecture remains open, and it is not even known if the limit in \eqref{limx-chowla} exists.  For instance, regarding the limiting behavior of $\E_{n \leq x} \lambda(n) \lambda(n+1)$, the best lower bound currently is 
$$ \liminf_{x \to \infty} \E_{n\leq x} \lambda(n) \lambda(n+1) \geq - 1 + \frac{1}{3}$$
due to Harman, Pintz, and Wolke \cite{hpw}, and the best upper bound is
$$ \limsup_{x \to \infty} \E_{n\leq x} \lambda(n) \lambda(n+1) \leq 1 - c$$
for some explicit constant $c>0$, from the breakthrough work of Matom\"aki and Radziwi{\l}{\l} \cite{mr}.  

While the Chowla conjecture would be expected to get more difficult as $k$ increases, the case of even $k$ (that is, an odd number of factors in \eqref{limx-chowla}) is slightly more tractable in some respects.  For instance, when $k$ is even it is no longer necessary to require that the shifts $h_0,\dots,h_k$ all be distinct, as it is not possible for the product of the linear factors to form a perfect square of a polynomial.  In \cite{elliott}, Elliott provided an elementary argument to show that
$$ \limsup_{x \to \infty} \left| \E_{n\leq x} \lambda(n) \lambda(n+1) \lambda(n+2)\right| \leq 1 - \frac{1}{21}$$
and in \cite[Corollary 1]{cass} it was shown (among other things) that
$$ \limsup_{x \to \infty} \left| \E_{n\leq x} \lambda(n) \lambda(n+1) \dots \lambda(n+2k)\right| \leq 1 - \frac{1}{3(k+1)}$$
for any $k \geq 1$.

In recent years, progress has been made on various averaged forms of Chowla's conjecture.  For instance, in \cite{mrt} Matom\"aki, Radziwi\l{}\l{} and Tao established a version of Chowla's conjecture where one performs some averaging in the $h_0,\dots,h_k$ parameters.  In this paper, we consider instead an averaged form in which the unweighted averages \eqref{avg-def} are replaced by logarithmic averages
$$ \E^{\log}_{a \in A} f(a) \coloneqq \frac{\sum_{a \in A} \frac{f(a)}{a}}{\sum_{a \in A} \frac{1}{a}}.$$
The \emph{logarithmically averaged Chowla conjecture} then asserts that for any $k \geq 0$ and any distinct integers $h_0,\dots,h_k$, one has
\begin{equation}\label{limx-chowla-log}
 \lim_{x \to \infty} \E^{\log}_{n \leq x} \lambda(n+h_0) \dots \lambda(n+h_k) = 0,
\end{equation}
or equivalently
$$
 \lim_{x \to \infty} \frac{1}{\log x} \sum_{n \leq x} \lambda(n+h_0) \dots \lambda(n+h_k) = 0.
$$
A simple summation by parts argument shows that for any choice of $k, h_0,\dots,h_k$, the ordinary Chowla conjecture \eqref{limx-chowla} implies its logarithmically averaged counterpart \eqref{limx-chowla-log}; however there does not seem to be any easy way to reverse this implication\footnote{For instance, as mentioned previously, for $k=0$ the Chowla conjecture is equivalent to the prime number theorem, whereas the logarithmically averaged case for $k=0$ can be proven by a short elementary argument that avoids use of the prime number theorem.  Also, if one replaces $\lambda(n)$ by the completely multiplicative function $n^{it}$ for some non-zero real $t$, one can easily check that $\E^{\log}_{n \leq x}n^{it}$ goes to zero as $x \to \infty$, but $\E_{n \leq x} n^{it}$ does not.}; see however \cite{tt-almostall}. 

In \cite{tao}, by introducing the \emph{entropy decrement argument} and combining it with the aforementioned work of Matom\"aki and Radziwi{\l}{\l} \cite{mr}, the first author established the logarithmically averaged Chowla conjecture in the $k=1$ case (i.e., for two-point correlations). In \cite{tao-higher}, a variant of the method was used to show that the logarithmically averaged Chowla conjecture was also equivalent to the logarithmically averaged form of the Sarnak conjecture \cite{sarnak}. It was shown by Frantzikinakis \cite{frantz-2} that this conjecture was also equivalent to the ergodicity of a certain family of dynamical systems which we will also encounter in this paper. We mention that very recently Frantzikinakis and Host \cite{fh2} proved the logarithmic Sarnak conjecture for all zero entropy topological dynamical systems that are uniquely ergodic.

It was also shown in \cite{tao-higher} that the logarithmic Chowla conjecture is equivalent to the local Gowers uniformity of the Liouville function over almost all short intervals (Conjecture 1.6 of that paper). It is only known that the Liouville function is Gowers uniform over long intervals (by \cite{gt-mobius}, \cite{gtz}), as opposed to short ones, and even the case of local $U^2$-uniformity of the Liouville function is a difficult open problem (also mentioned in \cite[Section 4]{tao}). The case of local $U^1$-uniformity over almost all very short intervals is already the Matom\"aki-Radziwi\l{}\l{} theorem \cite{mr} for the Liouville function. We manage to bypass these questions, and make no progress on them here, as we will only be working with odd order correlations of the Liouville function.

For any \emph{sign pattern} $\eps = (\eps_0,\dots,\eps_k) \in \{-1,+1\}^k$, let $A_\eps$ denote the set of natural numbers $n$ such that $\lambda(n+h) = \eps_h$ for $h=0,\dots,k$.  It is not difficult to show that if the Chowla conjecture holds for any number of shifts up to $k$, then all of the sets $A_\eps$ have natural density $\frac{1}{2^k}$, similarly, if the logarithmic Chowla conjecture holds for any number of shifts up to $k$, then all the sets $A_\eps$ have logarithmic density $\frac{1}{2^k}$.  In particular, either conjecture would imply that $A_\eps$ was infinite.  In \cite{hil}, Hildebrand showed that for any $k \leq 2$ and any sign pattern $\eps\in \{-1,+1\}^k$, the set $A_\eps$ was infinite; in \cite{mrt-2} Matom\"aki, Radziwi\l{}\l{} and Tao showed under the same hypotheses that $A_\eps$ in fact had positive lower density, and in \cite{kl2} Klurman and Mangerel showed that the upper logarithmic density of such sets was at least $1/28$.

As a consequence of the main results of this paper, we will be able to make further progress on the logarithmically averaged Chowla conjecture, and on the size of the sets $A_\varepsilon$.  We summarise (slightly simplified versions of) these applications as follows:

\begin{theorem}[New results towards the logarithmically averaged Chowla conjecture]\label{chow}
We have:

\begin{itemize}
\item[(i)]  (Odd order cases of the Chowla conjecture) For any even $k \geq 0$ (which corresponds to an odd number of shifts), and any integers $h_0,\dots,h_k$ (not necessarily distinct), the logarithmically averaged Chowla conjecture \eqref{limx-chowla-log} holds.
\item[(ii)]  (Liouville sign patterns of length three)  If $\eps \in \{-1,+1\}^3$, and $A_\eps$ is as above, then
$$ \lim_{x \to \infty} \E^{\log}_{n \leq x} 1_{A_\eps} = \frac{1}{8}.$$
\item[(ii)]  (Liouville sign patterns of length four)  If $\eps \in \{-1,+1\}^4$, and $A_\eps$ is as above, then
$$ \lim_{x \to \infty} \E^{\log}_{n \leq x} 1_{A_\eps} \geq \frac{1}{32}.$$
In particular, $A_\eps$ is infinite, thus all sixteen sign patterns of length four occur infinitely often in the Liouville function.
\end{itemize}
\end{theorem}

\subsection{More general multiplicative functions}

The results in Theorem \ref{chow} will be deduced from more general results concerning correlations of bounded multiplicative functions.
Define a \emph{$1$-bounded multiplicative function} to be a function $g: \N \to \D$ from the natural numbers $\N$ to the disk $\D \coloneqq \{z \in \C: |z| \leq 1 \}$ such that $g(nm) = g(n) g(m)$ whenever $n,m$ are coprime.  Thus for instance $\lambda$ is a $1$-bounded multiplicative function.

The limiting correlations
\begin{equation}\label{limx}
 \lim_{x \to \infty} \E_{n \leq x} g_0(n+h_0) \dots g_k(n+h_k)
\end{equation}
for $1$-bounded multiplicative functions $g_0,\dots,g_k$ and distinct shifts $h_0,\dots,h_k \in \Z$ (defining $g_0,\dots,g_k$ arbitrarily on non-positive integers) for some $k \geq 0$ have been extensively studied.  A well-known conjecture of Elliott \cite{elliott}, \cite{elliott2} asserts that the limit \eqref{limx} exists and is equal to zero unless each of the $g_j$ \emph{pretends} to be a twisted Dirichlet character $n \mapsto \chi_j(n) n^{it_j}$ in the sense that
\begin{equation}\label{ptj}
 \sum_p \frac{1 - \mathrm{Re}(g_j(p) \overline{\chi_j(p) p^{it_j}})}{p} < \infty
\end{equation}
for all $j=0,\dots,k$.  Specialising to the case when $g_0=\dots=g_k=\lambda$ recovers the Chowla conjecture \eqref{limx-chowla}.

The Elliott conjecture is known to hold for $k=0$ thanks to the work of Hal\'asz \cite{halasz}.  For $k \geq 1$, it was observed in
\cite{mrt} that this conjecture fails on a technicality; however one can repair the conjecture by replacing \eqref{ptj} with the stronger assumption\footnote{In the case where the functions $g_j$ are allowed to depend on the length $x$ of the average, one should strengthen assumption \eqref{ptj-weak} a bit further. However, we will not consider this variant of the conjecture here.}
\begin{equation}\label{ptj-weak}
\liminf_{X \to \infty} \inf_{|t| \leq X} \sum_{p \leq X} \frac{1 - \mathrm{Re}(g_j(p) \overline{\chi_j(p) p^{it}})}{p} < \infty
\end{equation}
for all $j=0,\ldots, k$. It was shown in \cite{mrt} that this repaired version of the Elliott conjecture holds if one is allowed to perform a non-trivial amount of averaging in the shifts $h_0,\dots,h_k$; we refer the reader to that paper for a precise statement.  The result in \cite{mrt} was generalised to averages over independent polynomials in several variables by Frantzikinakis in \cite{frantz}. The papers \cite{fh3}, \cite{matthiesen} in turn established two-dimensional variants of Elliott's conjecture. We also mention the recent paper of Klurman \cite{klurman} which gives an explicit formula for the limit \eqref{limx} in the ``pretentious'' case that \eqref{ptj} holds for all $j=0,\dots,k$.

A logarithmically averaged version the Elliott conjecture was considered by the first author in \cite{tao}.  
In that paper, an  ``entropy decrement'' argument was used to show that we have the logarithmically averaged two-point Elliott conjecture
$$
 \lim_{m \to \infty} \E^{\log}_{x_m/\omega_m \leq n \leq x_m} g_0(n+h_0) g_1(n+h_1) = 0,$$
or equivalently
$$
 \lim_{m \to \infty} \frac{1}{\log \omega_m} \sum_{x_m/\omega_m \leq n \leq x_m} \frac{g_0(n+h_0) g_1(n+h_1)}{n} = 0$$
for any $h_0\neq h_1$ and any sequences of numbers $1 \leq \omega_m \leq x_m$ that both go to infinity as $m \to \infty$, unless there exist Dirichlet characters $\chi_0,\chi_1$ such that \eqref{ptj-weak} holds for $j=0,1$.  

One technical difficulty in analysing these correlations is that it is not currently known whether the above limits exist.  To get around this problem, we shall (as in \cite{mrt-2}) work with \emph{generalised limit functionals}\footnote{As special cases of generalised limit functionals, one could consider \emph{Banach limits}, which also enjoy the shift-invariance property $\plim_{m \to \infty} a_{m+1} = \plim_{m \to \infty} a_m$, and \emph{ultrafilter limits}, which enjoy the homomorphism property $\plim_{m \to \infty} a_m b_m = \plim_{m \to \infty} a_m \plim_{m \to \infty} b_m$.  However it is not possible to satisfy both properties simultaneously for arbitrary sequences, and we will not use either of these properties in our arguments, so will work at the level of arbitrary generalised limit functionals.  The reader may wish to simply assume for a first reading that all (ordinary) limits of the sequences studied here converge, in which case one could replace these generalised limit functions by their ordinary counterparts.}  $\plim_{m \to \infty}: \ell^\infty(\N) \to \C$, which by definition are bounded linear functionals (of operator norm one) on the space $\ell^\infty(\N)$ of bounded sequences which extend the limit functional $\lim_{m \to \infty}: c_0(\N) \oplus \C \to \C$ on convergent sequences.  As is well known, the Hahn-Banach theorem may be used to (non-constructively) demonstrate that generalised limit functionals exist. Furthermore,
it is not hard to see that if a bounded sequence $x_m$ has the property that all of its generalised limits $\plim_{m \to \infty} x_m$ are equal to $\alpha$, then it converges to $\alpha$ in the ordinary limit. 

Given $1$-bounded multiplicative functions $g_0,\dots,g_k: \N \to \D$, shifts $h_0,\dots,h_k \in \Z$, sequences of real numbers $1 \leq \omega_m \leq x_m$ going to infinity, and a generalised limit functional $\plim$, one may form the \emph{correlation sequence} $f: \Z \to \D$ by the formula
\begin{equation}\label{fam}
f(a) \coloneqq \plim_{m \to \infty} \E^{\log}_{x_m/\omega_m \leq n \leq x_m} g_0(n+ah_0) \dots g_k(n+ah_k)
\end{equation}
for all integers $a$.  The logarithmically averaged (and corrected) version of the Elliott conjecture then asserts that when the $h_0,\dots,h_k$ are distinct, $f(a)$ vanishes for all non-zero $a$ unless there exist Dirichlet characters $\chi_0,\dots,\chi_k$ such that \eqref{ptj-weak} holds for all $j=0,\dots,k$; this is currently known to hold for $k=0,1$.

We do not settle this conjecture completely here.  However, we are able to obtain the following structural information on the correlation sequence $f$.  We say that a multiplicative function $g: \N \to \D$ \emph{weakly pretends} to be another multiplicative function $h: \N \to \D$ if one has
$$ \lim_{x\to \infty} \E^{\log}_{p \leq x} \left(1 - \mathrm{Re}(g(p) \overline{h(p)})\right) = 0$$
or equivalently
$$ \sum_{p \leq x} \frac{1 - \mathrm{Re}(g(p) \overline{h(p)})}{p} = o( \log\log x)$$
as $x \to \infty$.

We can now state our main theorem, which we prove in Section \ref{main-sec}.

\begin{theorem}[Structure of correlation sequences]\label{main}  Let $k \geq 0$, and let $h_0,\dots,h_k$ be integers and $g_0,\dots, g_k:\mathbb{N}\to \mathbb{D}$ any $1$-bounded multiplicative functions.  Let $1\leq \omega_m\leq  x_m$, $\plim$ and $f:\mathbb{N}\to \mathbb{D}$ be as above.  
\begin{itemize}
\item[(i)]  $f$ is the uniform limit of periodic functions $f_i$.
\item[(ii)]  If the product $g_0 \dots g_k$ does not weakly pretend to be $\chi$ for any Dirichlet character $\chi$, then $f$ vanishes identically.  
\item[(iii)]  If instead the product $g_0 \dots g_k$ weakly pretends to be a Dirichlet character $\chi$, then the periodic functions $f_i$ from part (i) can be chosen to be \emph{$\chi$-isotypic} in the sense that one has the identity $f_i(ab) = f_i(a) \chi(b)$ whenever $a$ is an integer and $b$ is an integer coprime to the periods of $f_i$ and $\chi$.
\end{itemize}
\end{theorem}

\begin{remark} Note that we do not require $h_0,\dots, h_k$ to be distinct in the main theorem. However, in the proof we may assume them to be distinct, since if $h_i=h_j$ for some $i$ and $j$, we may replace $g_i(n+h_i)g_j(n+h_j)$ with $g_ig_j(n+h_i)$.
\end{remark}

\begin{remark}  A simple special case of Theorem \ref{main} arises when the $g_0,\dots,g_k$ are all Dirichlet characters of a common period $q$.  In this case $f(a) = \E_{n \in \Z/q\Z} g_0(n+ah_0) \dots g_k(n+ah_k)$, and by using the substitution $n = bn'$ we obtain the isotopy property $f(ab) = f(a) \chi(b)$ whenever $b$ is coprime to $q$, where $\chi \coloneqq g_0 \dots g_k$.  Note that the period of $\chi$ may in fact be much smaller than $q$, however we would still only expect the isotopy property for $b$ coprime to $q$ in general.
\end{remark}

Naturally, part (i) of the theorem can be restated in the form that for any $\eps>0$ there is a periodic function $f_{\eps}:\mathbb{Z}\to \mathbb{C}$ such that $\sup_{a}|f(a)-f_{\eps}(a)|\leq \varepsilon$.

Among other things, (iii) implies that $f$ has the same parity as $\chi$, in the sense that $f(-a) = \chi(-1) f(a)$ for all integers $a$. 

\begin{remark}\label{rom}  Theorem \ref{main} can be generalised to cover more general correlation sequences of the form
$$ f(a) \coloneqq \plim_{m \to \infty} \E^{\log}_{x_m/\omega_m \leq n \leq x_m} g_0(q_1 n+ah_0) \dots g_k(q_k n+ah_k)$$
for fixed $q_1,\dots,q_k$; see Appendix \ref{mult-cor}.
\end{remark}

We have the following reformulation of part (ii) of Theorem \ref{main} that avoids all mention of generalised limits:

\begin{corollary}[A weak form of the logarithmically averaged Elliott conjecture]\label{dec}  Let $k \geq 0$, and let $g_0,\dots,g_k: \N \to \D$ be $1$-bounded multiplicative functions, such that the product $g_0 \dots g_k$ does not weakly pretend to be any Dirichlet character $\chi$.  Then for any  integers $h_0,\dots,h_k$ and any sequences of reals $1 \leq \omega_m \leq x_m$ that go to infinity, one has
$$ \lim_{m \to \infty} \E^{\log}_{x_m/\omega_m \leq n \leq x_m} g_0(n+h_0) \dots g_k(n+h_k) = 0.$$
In particular, the limit on the left-hand side exists.  Thus for instance one has
$$ \sum_{n \leq x} \frac{g_0(n+h_0) \dots g_k(n+h_k)}{n} = o(\log x)$$
as $x \to \infty$.
\end{corollary}

\begin{remark}  The hypothesis here neither implies nor is implied by the condition arising in the (repaired) Elliott conjecture, namely the failure of \eqref{ptj-weak} for some $j$.  For instance, consider the case when $g_j(n) = n^{it_j}$ for some real numbers $t_0,\dots,t_k$.  The (repaired) Elliott conjecture makes no prediction as to the asymptotic behaviour of the averages $\E^{\log}_{x_m/\omega_m \leq n \leq x_m} g_0(n+h_0) \dots g_k(n+h_k)$ in this case, since \eqref{ptj-weak} clearly holds for all $j$.  However, using the asymptotic $g_j(n+h_j) = (1+o(1)) n^{it_j}$ we see that these averages will go to zero if $\sum_{j=0}^k t_j \neq 0$, or equivalently if $g_0 \dots g_k$ does not weakly pretend to be any Dirichlet character.  This is consistent with Corollary \ref{dec} (or with Theorem \ref{main}).
\end{remark}

Specialising Theorem \ref{main} to the M\"obius and Liouville functions, we conclude
 
\begin{corollary}\label{chkc}  Let $k \geq 0$, and let $c_0,\dots,c_k$ be non-zero integers such that $c_0 + \dots + c_k$ is odd.  Then for any distinct integers $h_0,\dots,h_k$ and any sequences of reals $1 \leq \omega_m \leq x_m$ that go to infinity, one has
$$ \lim_{m \to \infty} \E^{\log}_{x_m/\omega_m \leq n \leq x_m} \mu^{c_0}(n+h_0) \dots \mu^{c_k}(n+h_k) = 0$$
and
$$ \lim_{m \to \infty} \E^{\log}_{x_m/\omega_m \leq n \leq x_m} \lambda^{c_0}(n+h_0) \dots \lambda^{c_k}(n+h_k) = 0.$$
\end{corollary}

From Corollary \ref{chkc}, we obtain the logarithmically averaged Chowla conjecture
$$ \E^{\log}_{x_m/\omega_m \leq n \leq x_m} \lambda(n+h_0) \dots \lambda(n+h_k) = o(1)$$
as $m \to \infty$ for all even values $k=0,2,4,\dots$ of $k$ (that is, for an \emph{odd} number of shifts).  More generally, by Remark \ref{rom}, we have
\begin{equation}\label{saa}
 \E^{\log}_{x_m/\omega_m \leq n \leq x_m} \lambda(q_0 n+h_0) \dots \lambda(q_k n+h_k) = o(1)
\end{equation}
for any $q_0,\dots,q_k \geq 1$ and $h_0,\dots,h_k \in \Z$, if $k$ is even (that is, if there is an odd number of shifts).  This particular consequence of our main theorem can be proven by a more direct and shorter argument, avoiding the use of ergodic theory; we recently gave such a simpler proof in \cite{tt-chowla}.  Specialising to the case $\omega_m = x_m = m$, we recover Theorem \ref{chow}(i).

\begin{remark}  Because the number of factors in \eqref{saa} is odd, we do not need to impose the non-degeneracy conditions $q_i h_j - q_j h_i \neq 0$.  It is conjectured that \eqref{saa} holds for odd $k$ as well, assuming this non-degeneracy condition.  
From the K\'atai-Bourgain-Sarnak-Ziegler criterion \cite{katai, bsz} (see also \cite{mv, dd}) one can show that if the claim \eqref{saa} holds for $2k-1$, then it holds for $k$; we omit the details.  This suggests that the case of even $k$ (which is what is proven in this paper) is easier than the case of odd $k$.
\end{remark}

Among other things, Corollary \ref{chkc} (together with the results of \cite{tao}) gives the conjectured logarithmic density of sign patterns for the Liouville and M\"obius functions of length up to three and four respectively:

\begin{corollary}\label{Sig}  Let $1 \leq \omega_m \leq x_m$ be sequences of reals that go to infinity.
\begin{itemize}
\item[(i)]  (Liouville sign patterns of length three)  If $k = 0,1,2$, $\eps_0,\dots,\eps_k \in \{-1,1\}$, and $A$ is the set of natural numbers $n$ such that $\lambda(n+j) = \eps_j$ for $j=0,\dots,k$, then
$$ \lim_{m \to \infty} \E^{\log}_{x_m/\omega_m \leq n \leq x_m} 1_A(n) = \frac{1}{2^{k+1}}.$$
In particular (specialising to the case $x_m=\omega_m=m$), we obtain Theorem \ref{chow}(ii).
\item[(ii)]  (M\"obius sign patterns of length four)  If $k = 0,1,2,3$, $\eps_0,\dots,\eps_k \in \{-1,0,1\}$, and $A$ is the set of natural numbers $n$ such that $\mu(n+j) = \eps_j$ for $j=0,\dots,k$, then
\begin{align*} \hspace{2cm}
&\lim_{m \to \infty} \E^{\log}_{x_m/\omega_m \leq n \leq x_m} 1_A(n)\\
& = \frac{1}{2^r}\cdot \lim_{x\to \infty}\frac{1}{x}|\{n\leq x:\,\, n+j\,\, \textnormal{squarefree for}\,\, 0\leq j\leq k\,\textnormal{iff}\,\, \varepsilon_j \neq 0\}|\\
&= C(\varepsilon_0,\ldots, \varepsilon_k),
\end{align*}
where $C(\varepsilon_{0},\ldots, \varepsilon_k)$ is an explicitly computable constant, and $r$ is the number of $0 \leq j \leq k$ for which $\eps_j \neq 0$.
\end{itemize}
\end{corollary}

We prove this result in Section \ref{applications}.  This improves upon the results in \cite{tao}, which handled the cases $k=0,1$, the results in \cite{mrt-2}, which showed that all Liouville sign patterns of length up to three and M\"obius sign patterns of length up to two occur with positive density, and the results in \cite{kl2}, which among other things gave the explicit lower bound of $1/28$ for the logarithmic upper density of length three sign patterns of the Liouville function.  Because the logarithmically averaged Chowla conjecture is now known for $k=0,1,2$, one can also establish that all Liouville sign patterns of length four occur with positive lower density, and more precisely one also obtains Theorem \ref{chow}(iii); this observation is due to Will Sawin (private communication) and also independently Kaisa Matom\"aki (private communication), and we present it in Section \ref{applications}.

\begin{remark}
In \cite{mrt-2}, the M\"obius function was more difficult to handle than the Liouville function, resulting in the need to study shorter sign patterns for the former than in the latter.  Here, the reverse seems to be true; but this is only because of the trivial fact that for any four consecutive numbers $n,n+1,n+2,n+3$, one of them must be divisible by four and thus be a zero of the M\"obius function.  In particular, our proof of Corollary \ref{Sig}(ii) breaks down when one replaces $n,n+1,n+2,n+3$ by non-consecutive numbers $n+h_0,n+h_1,n+h_2,n+h_3$ (unless they occupy distinct residue classes modulo $4$), whereas our proof of Corollary \ref{Sig}(i) can be easily adapted to the sign patterns of non-consecutive numbers $n+h_0,h+h_1,n+h_2$.
\end{remark}

\begin{remark}  Arguing as in the proof of \cite[Corollary 2.8]{mrt-2} (see also \cite{hil}, \cite{mr}), we conclude that the length five sign patterns $++++-$, $----+$, $-++++$, and $+----$ each occur for the Liouville function with positive lower logarithmic density. Arguing as in \cite[Proposition 2.9]{mrt-2}, we see that for any $k \geq 4$, the number $s(k)$ of length $k$ sign patterns that occur for the Liouville function  with positive upper logarithmic density satisfies $s(k+1)>s(k)$, but also from Theorem \ref{chow}(i) it follows that $s(k)$ is even (since the patterns $(\varepsilon_i)_{0\leq i\leq k-1}$ and $(-\varepsilon_i)_{0\leq i\leq k-1}$ can be computed to have the same logarithmic density), and thus $s(k)\geq 2k+8$ for $k\geq 4$. This improves slightly over the bound of $k+5$ from \cite{mrt-2}. In \cite[Theorem 1.2]{fh2}, it was shown that $s(k)$ grows faster than linearly with $k$.
\end{remark}

In Section \ref{applications} we will also establish the expected logarithmic density for arbitrarily long sign patterns for the number $\Omega(n)$ of prime factors relative to coprime moduli.  More precisely, we show the following.

\begin{corollary}\label{spin}  Let $1 \leq \omega_m \leq x_m$ be sequences of reals that go to infinity. Let $h_0,\dots,h_k$ be integers, and let $q_0,\dots,q_k$ be pairwise coprime natural numbers.  Then for any $\eps_j \in \Z/q_j\Z$ for $j=0,\dots,k$, one has
$$ \lim_{m \to \infty} \E^{\log}_{x_m/\omega_m \leq n \leq x_m} 1_A(n) = \frac{1}{q_0 \dots q_k},$$
where $A$ is the set of natural numbers $n$ such that $\Omega(n+h_j) = \eps_j\ (q_j)$ for $j=0,\dots,k$, where $\Omega(n)$ denotes the number of prime factors of $n$ (counting multiplicity).  Similarly if one replaces $\Omega(n)$ by $\omega(n)$, the number of prime factors of $n$ not counting multiplicity.
\end{corollary}

Note that the Chowla conjecture would follow if one could remove the hypothesis that the $q_0,\dots,q_k$ be coprime (and if we now insist on $h_0,\dots,h_k$ being distinct).\\

Finally, we have the following equidistribution results for ``linearly independent'' additive functions:

\begin{corollary}\label{addcor}  Let $f_0,\dots,f_k: \N \to \R/\Z$ be additive functions (thus $f_j(nm) = f_j(n) + f_j(m)$ whenever $j=0,\dots,k$ and $n,m$ are coprime).  Assume the following ``linear independence'' property: for any integer coefficients $a_0,\dots,a_k \in \Z$, not all zero one has
\begin{equation}\label{kip}
\limsup_{x \to \infty} \E^{\log}_{p \leq x} \| a_0 f_0(p) + \dots + a_k f_k(p) \|_{\R/\Z} > 0,
\end{equation}
where $\| \theta \|_{\R/\Z}$ denotes the distance from $\theta$ to the nearest integer.  Then, for any integers $h_0,\dots,h_k$, the vectors $(f_0(n+h_0), \dots, f_k(n+h_k))$ are asymptotically logarithmically equidistributed $\pmod 1$ in the following sense: for any sequences $1 \leq \omega_m \leq x_m$ that go to infinity, and any arcs $I_0,\dots,I_k \subset \R/\Z$, we have
$$ \lim_{m \to \infty} \E^{\log}_{x_m/\omega_m \leq n \leq x_m} 1_{I_0}(f_0(n+h_0)) \dots 1_{I_k}(f_k(n+h_k)) = |I_0| \dots |I_k|$$
where $|I|$ denotes the Lebesgue measure of $I$.  Furthermore, if $I_0,\dots,I_k$ are non-empty, then the set
$$ \{ n \in \N: f_0(n+h_0) \in I_0, \dots, f_k(n+h_k) \in I_k \}$$
has positive lower density.
\end{corollary}

In particular, we see that for any any positive integer $k$ and any real numbers $\alpha_1,\ldots, \alpha_k$ that are linearly independent over $\mathbb{Q}$, the vectors $(\alpha_1\Omega(n),\ldots, \alpha_k\Omega(n+k-1))$ are asymptotically logarithmically equidistributed $\pmod 1$. Furthermore, if we rewrite the statement of the above theorem in terms of multiplicative functions $g_j(n)=e^{2\pi i f_j(n)}$, then the case $k=1$ is connected with \cite[Conjecture 3.3]{klurman-mangerel} about equidistribution of shifts of multiplicative functions. However, Theorem \ref{addcor} does not completely resolve that conjecture, since we have a stronger non-pretentiousness assumption.  \\

Again, the result will be established in Section \ref{applications}.  When $k=0$, we see that if we had
$$\sum_{p} \frac{\| a_0 f_0(p) \|_{\R/\Z}}{p} < \infty$$
for some non-zero $a_0$, then by Wirsing's theorem \cite{wirsing} we would see that either $e(a_0 f_0(n))$ or $e(2a_0 f_0(n))$ has non-zero mean value, meaning that $f_0$ cannot be asymptotically equidistributed $\pmod 1$ in the ordinary sense.  However, one would not expect the condition \eqref{kip} to be  necessary for $k>0$.  If one wished to work with ordinary averages in Corollary \ref{addcor} rather than logarithmically averaged ones, one would need to modify \eqref{kip} to also rule out the possibility of $a_0 f_0 + \dots + a_k f_k$ pretending to behave like a function of the form $n \mapsto t \log n$ for some non-zero $t$ since such functions are not equidistributed if one does not perform logarithmic averaging.

\subsection{Proof ideas}

We now briefly discuss the methods of proof of Theorem \ref{main}.  The idea is to control the sequence $f(a)$ in two rather different ways. 

The first way to control $f(a)$ (carried out in Section \ref{sec: entropy} and in the beginning of Section \ref{main-sec}) starts with dilating the summation variable $n$ appearing in the correlation \eqref{fam} by primes $p$ in some dyadic range $2^r\leq p<2^{r+1}$ and then averaging over such primes using the multiplicativity of the functions $g_j$; this idea was also used in \cite{tao}. Here it is crucial that the average in the definition of $f(a)$ is a logarithmic one, with $\omega_m$ going to infinity. Denoting $G(n) \coloneqq g_0(n)\dots g_k(n)$ (and assuming for simplicity that the functions $g_j$ take values on the boundary of the unit disk), one then ends up with
\begin{align}\label{equ3}
f(a)=\mathbb{E}_{2^r\leq p<2^{r+1}}\overline{G(p)}\plim_{m\to \infty}\mathbb{E}_{x_m/\omega_m\leq n\leq x_m}^{\log}g_0(n+aph_0)\dots g_k(n+aph_k)p1_{p\mid n}+o_{r\to \infty}(1).
\end{align}
The factor $p1_{p\mid n}$ on the right-hand side is problematic, but has average value $1$ over $n$. By using the entropy decrement argument developed  in \cite{tao}, we can replace $p1_{p\mid n}$ with $1$ in \eqref{equ3} for ``many'' values of $r$. The original entropy decrement argument gave infinitely many values of $r$ for which the above works, but we need to prove a refined version of the argument that shows that the set of suitable $r$ contains almost all numbers with respect to logarithmic density. Thus we arrive at  the formula
\begin{align}\label{equ2}
f(a)=\plim_{m\to \infty}\mathbb{E}_{2^m\leq p<2^{m+1}}\overline{G(p)}f(ap)
\end{align} 
that is in some sense a functional equation for $f$ (with $\plim$ some suitable generalised limit functional).  In fact, by working more carefully, we can even obtain the stronger claim
\begin{align}\label{equ2a}
\plim_{m\to \infty}\mathbb{E}_{2^m\leq p<2^{m+1}} |\overline{G(p)}f(ap) - f(a)| = 0;
\end{align} 
see Theorem \ref{g-isotop} for a precise statement (very recently, a similar identity was utilised in \cite{fh2}). We remark that analysing quantities related to multiplicative functions via ''approximate functional equations'' is also utilised in Elliott's book \cite{elliott_book}. There, though, the results are mainly about mean values of multiplicative functions.\\

It will be convenient (following a suggestion of Maksym Radziwi{\l}{\l}) to iterate formula \eqref{equ2} to obtain
\begin{align}\label{equ5}
f(a)=\plim_{m_1\to \infty}\plim_{m_2\to \infty}\mathbb{E}_{2^{m_1}\leq p_1<2^{m_1+1}}\mathbb{E}_{2^{m_2}\leq p_2<2^{m_2+1}}\overline{G(p_1)}\, \overline{G(p_2)}f(ap_1p_2),
\end{align}
the advantage being that the right-hand side resembles a bilinear sum.

The other approach to analyzing $f$ proceeds via ergodic theory (see Sections \ref{sec: entropy} and \ref{sec: nilseq}). By the Furstenberg correspondence principle, we may find a probability space $(X,\mu)$ together with a measure-preserving invertible map $T:X\to X$ such that
\begin{align}\label{equ4}
f(a)=\int_{X}G_0(T^{ah_0}x)\dots G_k(T^{ah_k}x)\,d\mu(x)
\end{align}
for some bounded measurable functions $G_0,\dots, G_k: X\to \mathbb{D}$, whose precise form does not concern us. Making use of a recent result of Leibman \cite{le} on multiple correlations, \eqref{equ4} can be written as $f_1(a)+f_2(a)$, where $f_1$ is a \emph{nilsequence}  and $f_2$ is a sequence that \textit{goes to zero in uniform density} (see Section \ref{sec: nilseq} for definitions). Similarly as in the recent work of Le \cite{le}, we  see that the sequence $f_2$ goes to zero in uniform density also along the primes, implying that it is negligible in \eqref{equ5}, so the right-hand side of that formula becomes a bilinear sum of the nilsequence $f_1$. 

Using the theory of nilsequences (see Section \ref{sec: symb}), we can write any nilsequence $f_1$ up to any specified small error $\eps > 0$ as a linear combination of a periodic sequence $f_{0}$ (depending on $\eps$) and several ``irrational'' nilsequences. For irrational nilsequences, we can show that their bilinear averages tend to zero, so we end up with
\begin{align}
f(a)=\plim_{m_1\to \infty}\plim_{m_2\to \infty}\mathbb{E}_{2^{m_1}\leq p_1<2^{m_1+1}}\mathbb{E}_{2^{m_2}\leq p_2<2^{m_2+1}}\overline{G(p_1)}\,\overline{G(p_2)}f_{0}(ap_1p_2)+O(\eps),
\end{align}
which already shows part (i) of the main theorem, namely that $f$ is the uniform limit of periodic sequences. (These arguments are in the spirit of the structural theory of bounded multiplicative functions from \cite{fh}, which among other things indicates that bounded multiplicative functions can only resemble nilsequences if they are essentially periodic.)

 For part (ii) of the main theorem, we use the conclusion of part (i), which together with \eqref{equ2} gives
\begin{align*}
f_0(an)=\plim_{m\to \infty}\mathbb{E}_{2^m\leq p<2^{m+1}}\overline{G(p)}f_0(apn)+O(\varepsilon)
\end{align*} 
for some periodic function $f_0$ that approximates $f$ up to error  $\eps$. Multiplying both sides by an arbitrary Dirichlet character $\overline{\chi}(n)$, averaging over $n$ and using the fact that  $G(n)\overline{\chi}(n)$ does not weakly pretend to be $1$, we see that $\mathbb{E}_{n\leq x}f_0(an)\overline{\chi}(n)\ll \varepsilon$. This shows that the function $n\mapsto f(an)$ is orthogonal to all Dirichlet characters, but the only almost periodic function with this property is the identically zero function, giving the desired conclusion. The proof of part (iii) is quite similar to that of part (ii), but uses the information that $G(n)\overline{\chi'}(n)$ does not weakly pretend to be $1$ for any character $\chi'$ arising from a different primitive character than $\chi$.

\subsection{Acknowledgments}

TT was supported by a Simons Investigator grant, the James and Carol Collins Chair, the Mathematical Analysis \&
Application Research Fund Endowment, and by NSF grant DMS-1266164.  

JT was supported by UTUGS Graduate School and project number 293876 of the Academy of Finland.\\

Part of this paper was written while the authors were in residence at MSRI in spring 2017, which is supported by NSF grant DMS-1440140.  We thank Kaisa Matom\"aki for helpful discussions and encouragement, Maksym Radziwi{\l}{\l} for suggesting the use of semiprimes in the entropy decrement argument, Will Sawin for discussions on sign patterns of the Liouville function, and Bryna Kra for pointing us to the work of Leibman \cite{leibman} and Le \cite{le} on non-ergodic correlation sequences.  We thank the anonymous referees for their careful reading and for several corrections and suggestions.

\section{Notation}\label{notation-sec}

Unless otherwise specified, all sums, averages, and products over $p$ (or similar symbols such as $p_1$, $p_2$, etc.) will be understood to be over primes.

We use $e: \R/\Z \to \C$ to denote the standard character $e(x) \coloneqq e^{2\pi i x}$.  If $q$ is a natural number and $a,b$ are integers, we use $a\ (q) \in \Z/q\Z$ to denote the residue class of $a$ modulo $q$, and (by abuse of notation) $a=b\ (q)$ to denote claim that $a$ and $b$ have the same residue class modulo $q$.  We write $a|q$ if $a$ divides $q$, thus $a=b\ (q)$ if and only if $q|a-b$.  We observe that these definitions continue to make sense if $a,b$ take values in the profinite integers $\hat \Z$ (the inverse limit of the $\Z/q\Z$) rather than the integers $\Z$.

For technical reasons (having to do with our use of the \emph{nilcharacters} introduced in \cite{gtz}) we will need to deal with vector-valued sequences $f: \Z \to \C^m$ in addition to scalar sequences $f: \Z \to \C$.  We endow $\C^m$ with the usual Hilbert space norm
$$ \| (z_1,\dots,z_m) \|_{\C^m} \coloneqq \sqrt{|z_1|^2 + \dots + |z_m|^2}$$
and also recall that the tensor product $z \otimes w \in \C^{mn}$ of two vectors $z = (z_1,\dots,z_m) \in \C^m$, $w = (w_1,\dots,w_n) \in \C^n$ is defined by
$$ z \otimes w = (z_1 w_1, \dots, z_m w_1, \dots, z_1 w_n, \dots, z_m w_n)$$
thus for instance $\| z \otimes w \|_{\C^{mn}} = \|z\|_{\C^m} \|w\|_{\C^n}$.

Given a statement $S$, we use $1_S$ to denote the indicator of $S$, equal to $1$ when $S$ is true and $0$ when $S$ is false.  Given a set $E$, we use $1_E$ to denote its indicator function, thus $1_E(n) = 1_{n \in E}$.

We use the asymptotic notation $X \ll Y$, $Y \gg X$, or $X = O(Y)$ to denote the estimate $|X| \leq CY$ for some implied constant $C$, which is absolute unless otherwise specified.  If there is an asymptotic parameter $x$ going to infinity, we use $X = o(Y)$ to denote the estimate $|X| \leq c(x) Y$ where $c(x)$ goes to zero as $x$ goes to infinity (holding all other quantities not depending on $x$ fixed).

We will need to use probabilistic notation at various junctures of the paper.  Random variables (which are assumed to have some common probability space $\Omega$ as their common sample space) will be denoted in boldface to distinguish them from deterministic objects, e.g. we will be considering random functions ${\mathbf g}_j: \Z \to \D$ that are distinct from (but related to) their deterministic counterparts $g_j: \N \to \D$.  We use ${\mathbf P}(E)$ to denote the probability of an event $E$, and ${\mathbf E} \mathbf{X}$ to denote the expectation of a real or complex random variable $\mathbf{X}$. Further, we will need the Shannon entropy
\begin{align*}
\mathbf{H}(\mathbf{X}):=\sum_{X\in \mathcal{X}}\mathbf{P}(\mathbf{X}=X)\log\frac{1}{\mathbf{P}(\mathbf{X}=X)}
\end{align*}  
of a random variable $\mathbf{X}$ having a finite range $\mathcal{X}$, with the convention that $0\log\frac{1}{0}=0$. For two random variables $\mathbf{X},\mathbf{Y}$ with finite ranges $\mathcal{X}$ and $\mathcal{Y}$, we have the more general joint entropy
\begin{align*}
\mathbf{H}(\mathbf{X},\mathbf{Y}):=\sum_{X\in \mathcal{X}}\sum_{Y\in \mathcal{Y}}\mathbf{P}(\mathbf{X}=X,\mathbf{Y}=Y)\log\frac{1}{\mathbf{P}(\mathbf{X}=X,\mathbf{Y}=Y)}
\end{align*}
and the conditional entropy
\begin{align*}
\mathbf{H}(\mathbf{X}|\mathbf{Y})=\mathbf{H}(\mathbf{X},\textbf{Y})-\mathbf{H}(\mathbf{Y}).
\end{align*}
Lastly, we will make use of the concept of conditional mutual information
\begin{align*}
\mathbf{I}(\mathbf{X}:\mathbf{Y}|\mathbf{Z}):=\mathbf{H}(\mathbf{X}|\mathbf{Z})-\mathbf{H}(\mathbf{X}|\mathbf{Y},\mathbf{Z})
\end{align*}
of three random variables $\mathbf{X},\mathbf{Y},\mathbf{Z}$.

We will use the following standard arithmetic functions: 
\begin{itemize}
\item the number $\Omega(n)$ of prime factors of $n$ (counting multiplicity);
\item the number $\omega(n)$ of prime factors of $n$ (not counting multiplicity);
\item the Liouville function $\lambda(n) = (-1)^{\Omega(n)}$;
\item the M\"obius function $\mu(n)$, defined to equal $(-1)^{\omega(n)}$ when $n$ is squarefree, and $0$ otherwise;
\item the von Mangoldt function $\Lambda(n)$, defined to equal $\log p$ when $n$ is a power $p^j$ of a prime $p$ for some $j \geq 1$, and zero otherwise; and
\item the Euler totient function $\phi(n)$, defined as the number of invertible elements of $\Z/n\Z$.
\end{itemize}

\section{The Furstenberg correspondence principle and the entropy decrement argument}\label{sec: entropy}

Let the notation and hypotheses be as in Theorem \ref{main}.  By translating the $h_j$ (which is easily seen to not affect the generalised limit functionals) we may assume without loss of generality that the $h_j$ are all positive; in particular
\begin{equation}\label{hki}
h_0,\dots,h_k \in \{1,\dots,H\}
\end{equation}
for some natural number $H$, which we now fix.

In \cite{mrt-2}, generalised limit functionals were used to interpret correlations such as $f$ in the language of finitely additive probability theory.  In fact, thanks to tools such as the Riesz representation theorem and the Kolomogorov extension theorem, one can\footnote{Strictly speaking, one does not need the full strength of Proposition \ref{fcp} for the applications in this paper; one could instead establish Corollary \ref{fcor} below by adapting the proof of \cite[Proposition 2.1]{frantz-2}, and one could run the entropy argument non-asymptotically (as in \cite{tao}) without appeal to the correspondence principle, picking up some additional small error terms as a consequence.  We leave the details of this alternate argument to the interested reader.} interpret these correlations using the language of \emph{countably} additive probability theory:

\begin{proposition}[Furstenberg correspondence principle, probabilistic form] \label{fcp} Let the notation and hypotheses be as in Theorem \ref{main}.  Then there exists random functions ${\mathbf g}_0,\dots,{\mathbf g}_k: \Z \to {\mathbb{D}}$ and a random profinite\footnote{The \emph{profinite integers} $\hat \Z$ are defined as the inverse limit of the cyclic groups $\Z/q\Z$, with the weakest topology that makes the reduction maps $n \mapsto n\ (q)$ continuous.  This is a compact abelian group and thus has a well defined probability Haar measure.} integer ${\mathbf n} \in \hat \Z$, all defined on a common probability space $\Omega$, such that
\begin{equation}\label{ddd}
 {\mathbf E} F( ({\mathbf g}_i(h))_{0 \leq i \leq k, -N \leq h \leq N}, {\mathbf n} \ (q))
= \plim_{m \to \infty} \E^{\log}_{x_m/\omega_m \leq n \leq x_m} F\left( (g_i(n+h))_{0 \leq i \leq k, -N \leq h \leq N}, n\ (q)\right) 
\end{equation}
for any natural numbers $N, q$ and any continuous function $F: \D^{(k+1)(2N+1)} \times \Z/q\Z \to \R$. 
Furthermore, the random variables ${\mathbf g}_0,\dots,{\mathbf g}_k: \Z \to \D$ and ${\mathbf n} \in \hat \Z$ are a \emph{stationary process}, by which we mean that for any natural number $N$, the joint distribution of $({\mathbf g}_i(n+h))_{0 \leq i \leq k, -N \leq h \leq N} \in \D^{(k+1)(2N+1)}$ and $\mathbf{n} + n$ does not depend on $n$ as $n$ ranges across the integers.
\end{proposition}

\begin{proof}  Observe that for any natural numbers $N,q$ and continuous $F: \D^{(k+1)(2N+1)} \times \Z/q\Z \to \R$, the quantity
$$
\plim_{m \to \infty} \E^{\log}_{x_m/\omega_m \leq n \leq x_m} F\left( (g_i(n+h))_{0 \leq i \leq k, -N \leq h \leq N}), n\ (q) \right)$$
is unchanged if one replaces $n$ by $n+1$ in the summand.  This shows that any random variables ${\mathbf g}_0,\dots,{\mathbf g}_k, \mathbf{n}$ obeying \eqref{ddd} will also obey the identity
$$
 {\mathbf E} F\left( ({\mathbf g}_i(h))_{0 \leq i \leq k, -N \leq h \leq N}, \mathbf{n}\ (q)\right) = 
 {\mathbf E} F\left( ({\mathbf g}_i(h+1))_{0 \leq i \leq k, -N \leq h \leq N}, \mathbf{n}+1\ (q)\right) $$
for any continuous function $F: \D^{(k+1)(2N+1)} \to \R$, which implies that the ${\mathbf g}_0, \dots, {\mathbf g}_k$, $\mathbf{n}$ are a stationary process.  It thus suffices to construct random variables ${\mathbf g}_0,\dots,{\mathbf g}_k, {\mathbf n}$ obeying \eqref{ddd} for all $N,q$.  By the Kolmogorov extension theorem, it suffices to show that for each fixed choice of $N$ and $q$, there exist random variables obeying \eqref{ddd} for those values of $N$ and $q$.  But this easily follows from the Riesz representation theorem, since for fixed choices of $N$ and $q$ the right-hand side of \eqref{ddd} is clearly a positive linear functional on the space of continuous functions $F$ on the compact Hausdorff space $\D^{(k+1)(2N+1)} \times \Z/q\Z$.
\end{proof}

As a corollary of the countably additive probability theory interpretation, we can also interpret the correlation sequence $f(a)$ in the language of ergodic theory:

\begin{corollary}[Furstenberg correspondence principle, ergodic theory form]\label{fcor}  Let the notation and hypotheses be as in Theorem \ref{main}.  There exists a probability space $(X,\mu)$ together with a measure-preserving invertible shift $T: X \to X$, and measurable functions $G_0,\dots,G_k: X \to \D$, such that
$$ f(a) = \int_X G_0( T^{a h_0} x ) \dots G_k( T^{a h_k} x)\ d\mu(x)$$
for all integers $a$.
\end{corollary}

We remark that this result is essentially the same as that in \cite[Proposition 2.1]{frantz-2}.

\begin{proof}  Let $X = ({\mathbb{D}}^\Z)^{k+1}$ be the space of all tuples $(g_0,\dots,g_k)$ of functions $g_i: \Z \to {\mathbb{D}}$, equipped with the product topology and $\sigma$-algebra.  The distribution of the tuple $({\mathbf g}_0,\dots,{\mathbf g}_k)$ provided by the previous theorem is then a probability measure $\mu$ on $X$.  As this tuple is a stationary process, $X$ is invariant with respect to the shift map $T$ that maps any tuple $(g_0,\dots,g_k)$ to $(g_0(\cdot+1),\dots,g_k(\cdot+1))$.  Applying \eqref{ddd} with $q=1$ and $F$ of the form
$$ F( (g_{i,h})_{0 \leq i \leq k; -N \leq h \leq N} ) \coloneqq \prod_{j=0}^k g_{j, h_j}$$
for a sufficiently large $N$, we obtain the claim.
\end{proof}

\begin{remark} Such measure-preserving systems associated to multiplicative functions were also studied recently in \cite{frantz-2}, focusing in particular in the Liouville case $g_0 = \dots = g_k = \lambda$; the use of generalised limit functionals was avoided in that paper by assuming that all classical limits involved converge, but the analysis carries over to the generalised limit functional setting without difficulty.  A key technical point is that the measure-preserving system provided by the Furstenberg correspondence principle is \emph{not} known to be ergodic; indeed, in \cite{frantz-2} it was shown (in the Liouville case) that ergodicity is in fact equivalent to the full logarithmically averaged Chowla conjecture.  See also \cite{fh2} for some further analysis of these systems.
\end{remark}
 
In the next section we use Corollary \ref{fcor}, together with the ergodic-theory results of Bergelson-Host-Kra \cite{bhk}, Leibman \cite{leibman}, and Le \cite{le}, to approximate $f$ by nilsequences.

Let $G: \N \to \D$ denote the multiplicative function $G \coloneqq g_0 \dots g_k$.
In this section we will use the entropy decrement argument from \cite{tao} to show that the approximate identity
\begin{equation}\label{fap}
 f(ap) \approx f(a) G(p)
\end{equation}
holds in some sense for ``most'' integers $a$ and primes $p$.  

Fix $a$, and let $p$ be a prime.  From \eqref{fam}, we have
$$ f(a) G(p) = \plim_{m \to \infty} \frac{1}{\log \omega_m} \sum_{x_m/\omega_m \leq n \leq x_m} \frac{g_0(p) g_0(n+ah_0) \dots g_k(p) g_k(n+ah_k)}{n}.$$
From multiplicativity, we can write $g_j(p) g_j(n+ah_j)$ as $g_j(pn + aph_j)$ unless $n = - ah_j\ (p)$.   The latter case contributes $O\left(\frac{1}{p}\right)$ to the above generalised limit functional for each $j$.  Thus we have
$$ f(a) G(p) = \plim_{m \to \infty} \frac{1}{\log \omega_m} \sum_{x_m/\omega_m \leq n \leq x_m} \frac{g_0(pn+aph_0) \dots g_k(pn+aph_k)}{n} + O\left(\frac{1}{p}\right)$$
where we henceforth allow implied constants in the asymptotic notation to depend on $k$.  If we now make $pn$ rather than $n$ the variable of summation, we conclude that
$$ f(a) G(p) = \plim_{m \to \infty} \frac{1}{\log \omega_m} \sum_{px_m/\omega_m \leq n \leq px_m} \frac{g_0(n+aph_0) \dots g_k(n+aph_k) p1_{p|n} }{n} + O\left(\frac{1}{p}\right).$$
We can adjust the range of $n$ from $p x_m / \omega_m \leq n \leq p x_m$ to $x_m / \omega_m \leq n \leq x_m$ without affecting the generalised limit functional, since $\omega_m$ goes to infinity.  Thus
$$ f(a) G(p) = \plim_{m \to \infty} \frac{1}{\log \omega_m} \sum_{x_m/\omega_m \leq n \leq x_m} \frac{g_0(n+aph_0) \dots g_k(n+aph_k) p1_{p|n} }{n} + O\left(\frac{1}{p}\right).$$
Comparing this with \eqref{fam} (with $a$ replaced by $ap$), we conclude that
\begin{align*}& f(a) G(p) - f(ap)\\
&= \plim_{m \to \infty} \frac{1}{\log \omega_m} \sum_{x_m/\omega_m \leq n \leq x_m} \frac{g_0(n+aph_0) \dots g_k(n+aph_k) (p1_{p|n}-1) }{n} + O\left(\frac{1}{p}\right),\end{align*}
and hence by Proposition \ref{fcp} we have the formula
$$ f(a) G(p) - f(ap) = {\mathbf E} \mathbf{g}_0(aph_0) \dots \mathbf{g}_k(aph_k) (p 1_{p|\mathbf{n}} - 1) + O\left(\frac{1}{p}\right).$$
If we define $c_p$ to be the signum of $f(a) G(p) - f(ap)$, then $|c_p| \leq 1$ for all $p$, and we have
$$ |f(a) G(p) - f(ap)| = {\mathbf E} c_p \mathbf{g}_0(aph_0) \dots \mathbf{g}_k(aph_k) (p 1_{p|\mathbf{n}} - 1) + O\left(\frac{1}{p}\right).$$
In order to apply the entropy decrement argument introduced in \cite{tao} it is convenient to discretise the random functions $\mathbf{g}_j$, so that they only take finitely many values.  Fix a small parameter $\eps>0$.  Let $\mathbf{g}_{j,\eps}(n)$ be the random variable formed by rounding $\mathbf{g}_j(n)$ to the nearest Gaussian integer multiple of $\eps$ (breaking ties using (say) the lexicographical ordering on the Gaussian integers), then by the triangle inequality we have
$$ |f(a) G(p) - f(ap)| = {\mathbf E} c_p \mathbf{g}_{0,\eps}(aph_0) \dots \mathbf{g}_{k,\eps}(aph_k) (p 1_{p|\mathbf{n}} - 1) + O(\eps)$$
if $p$ is sufficiently large depending on $\eps$.  By stationarity, we then also have
$$ |f(a) G(p) - f(ap)| = {\mathbf E} c_p \mathbf{g}_{0,\eps}(l + aph_0)) \dots \mathbf{g}_{k,\eps}(l+aph_k) (p 1_{\mathbf{n} = -l\ (p)} - 1) + O(\eps)$$
for all integers $l$.
Averaging over a dyadic range $2^m \leq p < 2^{m+1}$ and $1 \leq l \leq 2^m$, we conclude that
\begin{equation}\label{chirp}
 {\mathbb E}_{2^m \leq p < 2^{m+1}} |f(a) G(p) - f(ap)| = {\mathbf E} F_m( \mathbf{X}_m, \mathbf{Y}_m ) + O( \eps )
\end{equation}
for $m$ sufficiently large depending on $\eps$, where 
\begin{itemize}
\item[(i)] $\mathbf{X}_m \in \D^{(k+1)2^{m+2} H'}$ is the random variable
\begin{equation}\label{xmdef}
 \mathbf{X}_m \coloneqq \left( \mathbf{g}_{j,\eps}( l ) \right)_{0 \leq j \leq k; 1 \leq l \leq 2^{m+2} H'}
\end{equation}
with $H' \coloneqq (1+|a|) H$ and $H$ is as in \eqref{hki};
\item[(ii)] $\mathbf{Y}_m \in \prod_{2^m \leq p < 2^{m+1}} \Z/p\Z$ is the random variable
$$ \mathbf{Y}_m \coloneqq \left( \mathbf{n}\ (p) \right)_{2^m \leq p < 2^{m+1}};$$
and
\item[(iii)] $F_m: \D^{(k+1)2^{m+2} H'} \times \prod_{2^m \leq p < 2^{m+1}} \Z/p\Z \to \C$ is the function
\begin{align*}
& F_m( (g_{j,l})_{0 \leq j \leq k; 1 \leq l \leq 2^{m+2} H'}, (n_p)_{2^m \leq p < 2^{m+1}} )\\
&\quad \coloneqq \E_{1 \leq l \leq 2^m} \E_{2^m \leq p < 2^{m+1}} c_p g_{0,l+aph_0} \dots g_{k,l+aph_k} (p 1_{n_p = -l\ (p)} - 1).
\end{align*}
\end{itemize}

From \eqref{ddd} and the Chinese remainder theorem, we see that the $\mathbf{Y}_m$ are uniformly distributed in $\prod_{2^m \leq p < 2^{m+1}} \Z/p\Z$, and are jointly independent in $m$.  However, they may correlate with the $\mathbf{X}_m$.  We ignore this issue for the moment by introducing a new random variable $\mathbf{U}_m = (\mathbf{u}_p)_{2^m \leq p < 2^{m+1}}$, drawn uniformly at random from $\prod_{2^m \leq p < 2^{m+1}} \Z/p\Z$.  For any deterministic vector
$$ X_m = \left( g_{j,l} \right)_{0 \leq j \leq k; 1 \leq l \leq 2^{m+2} H'} \in \D^{(k+1)2^{m+2} H'},$$
we can expand
$$ F_m(X_m, \mathbf{U}_m) = \E_{2^m \leq p < 2^{m+1}} \mathbf{Z}_p$$
where $\mathbf{Z}_p$ is the random variable
$$ \mathbf{Z}_p \coloneqq \E_{1 \leq l \leq 2^m} c_p g_{0,l+ph_0} \dots g_{k,l+ph_k} (p 1_{\mathbf{u}_p = -l\ (p)} - 1).$$
The random variables $\mathbf{Z}_p$ are clearly jointly independent in $p$. Since $\mathbf{u}_p$ is uniformly distributed in $\Z/p\Z$, we also see that each random variable $\mathbf{Z}_p$ has mean zero.  One easily verifies from the triangle inequality that $\mathbf{Z}_p$ is bounded in magnitude by $O(1)$.  Applying Hoeffding's inequality \cite{hoeffding}, together with the prime number theorem, we conclude the concentration of measure estimate
\begin{equation}\label{xm}
 \P( |F_m(X_m, \mathbf{U}_m)| \geq \eps) \ll \exp( - c \eps^2 2^m / m ) 
\end{equation}
for some absolute constant $c>0$.

To pass from $F_m(X_m, \mathbf{U}_m)$ back to $F_m( \mathbf{X}_m, \mathbf{Y}_m)$, we use the following information-theoretic inequality (cf. \cite[Lemma 3.3]{tao}).  For the basic definitions and properties of information-theoretic quantities such as Shannon entropy and mutual information that we will need, see \cite[\S 3]{tao} and Section \ref{notation-sec}.

\begin{lemma}  Let $\mathbf{Y}$  be a random variable taking values in a finite non-empty set ${\mathcal Y}$, and let $\mathbf{U}$ be a uniform random variable taking values in the same set.  Then for any subset $E$ of ${\mathcal Y}$, one has
$$ \mathbf{P}( \mathbf{Y} \in E ) \leq \frac{\mathbf{H}( \mathbf{U} ) - \mathbf{H}( \mathbf{Y} ) + \log 2}{\log \frac{1}{\mathbf{P}( \mathbf{U} \in E )} },$$
where the Shannon entropy $\mathbf{H}(\cdot)$ is defined in Section \ref{notation-sec}.
\end{lemma}

\begin{proof} We evaluate the conditional entropy $\mathbf{H}( \mathbf{Y} | 1_{\mathbf{Y} \in E} )$ in two different ways.  On one hand, we have
\begin{align*}
\mathbf{H}( \mathbf{Y} | 1_{\mathbf{Y} \in E} ) &= \mathbf{H}( \mathbf{Y}, 1_{\mathbf{Y} \in E} ) - \mathbf{H}( 1_{\mathbf{Y} \in E} ) \\
&= \mathbf{H}( \mathbf{Y} ) - \mathbf{H}( 1_{\mathbf{Y} \in E} )  \\
&\geq \mathbf{H}( \mathbf{Y} ) - \log 2
\end{align*}
thanks to Jensen's inequality.  On the other hand, by a further application of Jensen's inequality one has
\begin{align*}
\mathbf{H}( \mathbf{Y} | 1_{\mathbf{Y} \in E} ) &= \mathbf{P}( \mathbf{Y} \in E ) \mathbf{H}( \mathbf{Y} | \mathbf{Y} \in E )
+ \mathbf{P}( \mathbf{Y} \not \in E ) \mathbf{H}( \mathbf{Y} | \mathbf{Y} \not \in E ) \\
&\leq \mathbf{P}( \mathbf{Y} \in E ) \log |E| + (1 - \mathbf{P}( \mathbf{Y} \in E )) \log |{\mathcal Y}| \\
&= \log |{\mathcal Y}| - \mathbf{P}( \mathbf{Y} \in E ) \log \frac{|{\mathcal Y}|}{|E|} \\
&= \mathbf{H}( \mathbf{U} ) - \mathbf{P}( \mathbf{Y} \in E ) \log \frac{1}{\mathbf{P}( \mathbf{U} \in E )}.
\end{align*}
Combining the two bounds, we obtain the claim.
\end{proof}

Taking $\mathbf{U}=\mathbf{U}_m$ and $E=\{Y:\, |F_m(X_m,Y)|\geq \varepsilon\}$ in this lemma and recalling \eqref{xm}, we conclude that for $m$ sufficiently large depending on $a,\eps$, one has
\begin{equation}\label{fmy}
 \P( |F_m(X_m, \mathbf{Y})| \geq \eps) \ll \eps 
\end{equation}
whenever $\mathbf{Y}$ is a random variable taking values in $\prod_{2^m \leq p < 2^{m+1}} \Z/p\Z$ which has sufficiently high entropy in the sense that
$$ \mathbf{H}( \mathbf{U}_m ) - \mathbf{H}( \mathbf{Y} ) \ll \eps^3 \frac{2^m}{m}.$$

To use this, let $m_0$ be a sufficiently large natural number depending on $a,\eps$.  For $m \geq m_0$, let $\mathbf{Y}_{<m}$ denote the random variable $(\mathbf{Y}_{m'})_{m_0 \leq m' < m}$.  By \eqref{ddd} and the Chinese remainder theorem, the $\mathbf{Y}_m$ are uniformly distributed in $\prod_{2^m \leq p < 2^{m+1}} \Z/p\Z$ and are jointly independent in $m$.  In particular we have the conditional entropy identity
\begin{equation}\label{ssa}
 \mathbf{H}( \mathbf{Y}_m | \mathbf{Y}_{<m} = Y_{<m} ) = \mathbf{H}( \mathbf{U}_m )
\end{equation}
for any value $Y_{<m}$ in the range of $\mathbf{Y}_{<m}$.
Now suppose for the moment that we have for the conditional mutual information (defined in Section \ref{notation-sec}) the bound
\begin{equation}\label{imi}
 \mathbf{I}( \mathbf{X}_m : \mathbf{Y}_m | \mathbf{Y}_{<m} ) \leq \eps^4 \frac{2^m}{m},
\end{equation}
which roughly speaking asserts some weak conditional independence between the random variables $\mathbf{X}_m$ and $\mathbf{Y}_m$ relative to $\mathbf{Y}_{<m}$. We compute
\begin{align*}
&\mathbf{I}(\mathbf{X}_m:\mathbf{Y}_m|\mathbf{Y}_{<m})=\mathbf{I}(\mathbf{Y}_m:\mathbf{X}_m|\mathbf{Y}_{<m})=\mathbf{H}(\mathbf{Y}_m|\mathbf{Y}_{<m})-\mathbf{H}(\mathbf{Y}_{m}|\mathbf{X}_m,\mathbf{Y}_{<m})\\
&=\sum_{X_m, Y_{<m}} \mathbf{P}( \mathbf{X}_m = X_m, \mathbf{Y}_{<m} = Y_{<m} )
(\mathbf{H}( \mathbf{Y}_m | \mathbf{Y}_{<m} = Y_{<m} ) - \mathbf{H}( \mathbf{Y}_m | \mathbf{X}_m = X_m, \mathbf{Y}_{<m} = Y_{<m} )),
\end{align*}
where $X_m, Y_{<m}$ vary over the ranges of $\mathbf{X}_m, \mathbf{Y}_{<m}$ respectively.  Thus by Markov's inequality, we see that with probability $1-O(\eps)$, the pair $(\mathbf{X}_m, \mathbf{Y}_{<m})$ attains a value $(X_m, Y_{<m})$ such that
$$ \mathbf{H}( \mathbf{Y}_m | \mathbf{Y}_{<m} = Y_{<m} ) - \mathbf{H}( \mathbf{Y}_m | \mathbf{X}_m = X_m, \mathbf{Y}_{<m} = Y_{<m} ) \ll \eps^3 \frac{2^m}{m}$$
and thus by \eqref{fmy} and \eqref{ssa} (applied to $\mathbf{Y}_{<m}$ after conditioning to the event $\mathbf{X}_m = X_m$) we have the conditional probability bound
$$ \P( |F_m( \textbf{X}_m, \mathbf{Y}_m )|\geq \varepsilon | \mathbf{X}_m = X_m, \mathbf{Y}_{<m} = Y_{<m} ) \ll \eps.$$
Multiplying by $\P( \mathbf{X}_m = X_m, \mathbf{Y}_{<m} = Y_{<m} )$ and summing in $X_m, Y_{<m}$, we conclude that
$$ \textbf{E} F_m( \mathbf{X}_m, \mathbf{Y}_m ) \ll \eps $$
whenever $m \geq m_0$ is such that \eqref{imi} holds.   Combining this with \eqref{chirp}, we conclude that
$$ {\mathbb E}_{2^m \leq p < 2^{m+1}} |f(a) G(p) - f(ap)| \ll \eps$$
whenever $m \geq m_0$ is such that \eqref{imi} holds.

Call a set ${\mathcal N}$ of natural numbers \emph{log-small} if
$$\sum_{m \in {\mathcal N}: m \leq x} \frac{1}{m} = o_{x \to \infty}(\log x)$$
as $x \to \infty$, and \emph{log-large} otherwise.
We say that an assertion $P(m)$ holds for \emph{log-almost all} $m$ if it holds for $m$ outside of a log-small set.  We then have

\begin{proposition}[Entropy decrement argument]  The bound \eqref{imi} holds for log-almost all $m \geq m_0$.
\end{proposition}

\begin{proof}
For $m \geq m_0$, consider the quantity
$$ \mathbf{H}( \mathbf{X}_{m+1} | \mathbf{Y}_{<m+1} ).$$
One can split $\mathbf{X}_{m+1}$ as the concatenation of $\mathbf{X}_m$ and $\mathbf{X}'_m$, where
$$ \mathbf{X}'_m \coloneqq \left( \mathbf{g}_{j,\eps}( a (l + 2^{m+2} H) ) \right)_{0 \leq j \leq k; 1 \leq l \leq 2^{m+2} H}$$
is a translated version of $\mathbf{X}_m$.
Thus by the Shannon entropy inequalities we have
$$ \mathbf{H}( \mathbf{X}_{m+1} | \mathbf{Y}_{<m+1} ) \leq \mathbf{H}( \mathbf{X}_{m} | \mathbf{Y}_{<m+1} ) + \mathbf{H}( \mathbf{X}'_{m} | \mathbf{Y}_{<m+1} ) .$$
By stationarity (and the fact that translating $\mathbf{n}$ does not affect the $\sigma$-algebra generated by $\mathbf{Y}_{<m+1}$) we have
$$ \mathbf{H}( \mathbf{X}'_{m} | \mathbf{Y}_{<m+1} )  = \mathbf{H}( \mathbf{X}_{m} | \mathbf{Y}_{<m+1} ) $$
and thus
$$ \mathbf{H}( \mathbf{X}_{m+1} | \mathbf{Y}_{<m+1} ) \leq 2 \mathbf{H}( \mathbf{X}_{m} | \mathbf{Y}_{<m+1} ).$$
Since $\mathbf{Y}_{<m+1}$ is the concatenation of $\mathbf{Y}_{<m}$ and $\mathbf{Y}_m$, we have the identity
$$ \mathbf{H}( \mathbf{X}_{m} | \mathbf{Y}_{<m+1} ) = \mathbf{H}( \mathbf{X}_{m} | \mathbf{Y}_{<m} ) - \mathbf{I}( \mathbf{X}_m : \mathbf{Y}_m | \mathbf{Y}_{<m} )$$
and hence
$$ \frac{\mathbf{H}( \mathbf{X}_{m+1} | \mathbf{Y}_{<m+1} )}{2^{m+1}} \leq \frac{\mathbf{H}( \mathbf{X}_{m} | \mathbf{Y}_{<m} )}{2^m} - \frac{\mathbf{I}( \mathbf{X}_m : \mathbf{Y}_m | \mathbf{Y}_{<m} )}{2^m}.$$
Telescoping this, and using the non-negativity of conditional entropy, we see that
$$ \sum_{m \geq m_0} \frac{\mathbf{I}( \mathbf{X}_m : \mathbf{Y}_m | \mathbf{Y}_{<m} )}{2^m} < \infty $$
and hence
$$ \sum_{m \geq m_0: \mathbf{I}( \mathbf{X}_m : \mathbf{Y}_m | \mathbf{Y}_{<m} ) \geq \eps^4 2^m / m} \frac{1}{m} < \infty$$
and the claim follows.
\end{proof}

Combining this proposition with the previous analysis, we have established the following precise version of \eqref{fap}:

\begin{theorem}[Approximate $G$-isotopy]\label{g-isotop}  Let the notation and hypotheses be as in Theorem \ref{main}.  For any integer $a$ and $\eps>0$, one has
$$  {\mathbb E}_{2^m \leq p < 2^{m+1}} |f(a) G(p) - f(ap)| \ll \eps$$
for log-almost all natural numbers $m$, where $G \coloneqq g_0 \dots g_k$.
\end{theorem}

\begin{remark} The above arguments in fact give a slightly stronger claim, namely the set ${\mathcal M}$ of exceptional $m$ is not just log-small, but obeys the more precise estimate
$$ \sum_{m \in {\mathcal M}} \frac{1}{m} \ll \eps^{-4} \log \frac{1}{\eps}$$
where the implied constant can depend on $a,k,H$.  This can lead to a slight strengthening of Theorem \ref{main}, in which the notion of one multiplicative function weakly pretending to be another is strengthened in a somewhat complicated fashion, but we will not attempt to make this explicit here.
\end{remark}

\section{Nilsequence theory}\label{sec: nilseq}

To proceed further, we will invoke some deep ergodic theory results on recurrence sequences to obtain additional control on the sequence $f$. More precisely, we will show  that $f$ behaves like a \emph{nilsequence}, which turns out to be incompatible with the approximate isotopy property of $f$ from Theorem \ref{g-isotop}, unless $f$ behaves like a periodic sequence (cf. the structural theory of multiplicative functions in \cite{fh}). We start by defining nilsequences.

\begin{definition}[Nilsequence]  Let $d \geq 1$ be a natural number.  A \emph{filtered group} $G = (G,G_\bullet)$ of degree $\leq d$ is a group $G$, together with a sequence $G_\bullet  = (G_i)_{i=0}^\infty$ of nested subgroups
$$ G = G_0 \supset G_1 \supset \dots $$
with $G_{d+1} = \{\mathrm{id}\}$, and $[G_i, G_j] \subset G_{i+j}$ for all $i,j \geq 0$.  By a \emph{polynomial sequence} adapted to a filtered group, we mean a map $g: \Z \to G$ such that $\partial_{h_i} \dots \partial_{h_i} g(n) \in G_i$ for all $i \geq 0$, $n \in \Z$, and $h_1,\dots,h_i \in \Z$, where $\partial_h g(n) \coloneqq g(n+h) g(n)^{-1}$.  A \emph{filtered nilmanifold} $G/\Gamma = (G/\Gamma, G, G_\bullet)$ of degree $\leq d$ is a filtered connected, simply connected nilpotent Lie group $(G,G_\bullet)$ of degree $s$ (with all subgroups $G_i$ also connected, simply connected nilpotent Lie groups), together with a quotient $G/\Gamma$ of $G$ by a lattice $\Gamma$ (that is to say, a discrete and cocompact subgroup) of $G$, such that $\Gamma_i \coloneqq \Gamma \cap G_i$ is a lattice of $G_i$ for all $i \geq 0$.

A (vector-valued) \emph{basic nilsequence} of degree $\leq d$ is a function $f: \Z \to \C^m$ for some $m \geq 1$ of the form $f(n) \coloneqq F( g(n) )$, where $G/\Gamma$ is a filtered nilmanifold of degree $d$, $g$ is a polynomial sequence adapted to the corresponding filtered group $(G, G_\bullet)$, and $F: G \to \C^m$ is a smooth continuous function which is \emph{$\Gamma$-automorphic} in the sense that $F(g\gamma) = F(g)$ for all $g \in G$ and $\gamma \in \Gamma$.  A \emph{nilsequence} of degree $\leq d$ is a uniform limit of basic nilsequences of degree $\leq d$.  If $m=1$, we say that the nilsequence is \emph{scalar-valued}.
\end{definition}

\begin{example}  If $P_1,\dots,P_k: \Z \to \R/\Z$ are a finite number of polynomials of degree at most $d$, and $F: (\R/\Z)^k \to \C^m$ is a smooth function, then the sequence $n \mapsto F( P_1(n), \dots, P_k(n) )$ is a basic nilsequence of degree $\leq d$, which is scalar-valued if $m=1$.  In particular, for any real numbers $\alpha_0,\dots,\alpha_d \in \R$, the sequence $n \mapsto e( \alpha_d n^d + \dots + \alpha_0)$ is a scalar-valued basic nilsequence of degree $\leq d$.  If $P_1,P_2,\dots: \Z \to \R/\Z$ are an infinite sequence of polynomials of degree at most $d$, and $c_1,c_2,\dots$ are an absolutely summable sequence of elements of $\C^m$, then $\sum_{j=1}^\infty c_j e( P_j(n) )$ is a nilsequence of degree $\leq d$.

There are more exotic examples of nilsequences that arise from ``bracket polynomials'' rather than genuine polynomials.  For instance, if $F: (\R/\Z)^2 \to \C^m$ is a smooth function that vanishes in a neighbourhood of the axis $\{ (x,0): x \in \R/\Z \}$, and $\alpha,\beta$ are real numbers, then the sequence
$$ n \mapsto F( \{ \alpha n \} \beta n \,\, \textnormal{mod}\,\, 1, \alpha n \,\, \textnormal{mod}\,\, 1 )$$
is a basic nilsequence of degree $\leq 2$, where $\{ \}$ denotes the fractional part function; the vanishing of $F$ near the axis is needed to be able to smoothly represent the Heisenberg nilmanifold $\begin{pmatrix} 1 & \R & \R \\ 0 & 1 & \R \\ 0 & 0 & 1 \end{pmatrix} / 
\begin{pmatrix} 1 & \Z & \Z \\ 0 & 1 & \Z \\ 0 & 0 & 1 \end{pmatrix}$ by a single coordinate chart that avoids the singularities of the fractional part function $\{ \alpha n\}$.  See \cite[\S 6]{gtz} for details of this construction.
\end{example}

\begin{remark}  The original definition of a nilsequence in \cite{bhk} assumed $F$ was merely continuous rather than smooth, and used linear sequences $n \mapsto g^n g_0$ rather than polynomial sequences.  However, the extra imposition of smoothness is fairly harmless since the Stone-Weierstrass theorem ensures that any continuous automorphic $F$ can be uniformly approximated by smooth automorphic $F$, and also polynomial orbits can always be lifted to linear orbits in a larger nilmanifold (see \cite[Appendix C]{gtz}).
\end{remark}

The way nilsequences will enter into our analysis is via multiple recurrence sequences, and in particular from the following result:

\begin{theorem}\label{decomp-erg}  Let $(X,\mu)$ be a probability space, and let $T: X \to X$ be a measure-preserving action on this space.  Let $G_0,\dots,G_k \in L^\infty(X)$, and let $h_0,\dots,h_{k}$ be integers for some $k \geq 0$.  Then we have a decomposition
$$ \int_X G_0(T^{h_0 n} x) \dots G_k(T^{h_k n} x)\ d\mu(x) = f_1(n) + f_2(n)$$
for all $n \in \Z$, where $f_1: \Z \to \C$ is a nilsequence of degree $\leq D$ for some $D$, and $f_2: \Z \to \C$ is a sequence that goes to zero in uniform density, in the sense that
\begin{equation}\label{f2}
\lim_{N \to \infty} \sup_M {\mathbb E}_{M \leq n < M+N}  |f_2(n)| = 0
\end{equation}
\end{theorem}
 
\begin{proof}  By concatenating functions $G_j$ with a common value of $h_j$, we may assume that the $h_j$ are all distinct.

For the first claim, see \cite[Theorem 5.2]{leibman} (which in fact handled the more general situation of polynomial shifts); the case when $X$ was ergodic and $h_i=i$ was previously obtained in \cite[Theorem 1.9]{bhk}.  It is likely that the analysis in \cite{leibman} allows us to take $D=k$, but we will not need this bound here.
\end{proof}

For our application, it will be important to localise the zero density claim \eqref{f2} to multiples of primes:

\begin{proposition}\label{trap}  If $f_2$ is the function arising in Theorem \ref{decomp-erg}, then
$$ \lim_{x \to \infty} {\mathbb E}_{p \leq x} |f_2(ap)| = 0$$
for any fixed nonzero integer $a$.
\end{proposition}

We remark that results of this type were recently obtained by Le \cite{le}.

\begin{proof}  We allow all implied constants to depend on $a, k$, $G_0,\dots,G_k$, $f_1$, $f_2$.  Note that as all nilsequences are bounded, $f_2$ must also be bounded, thus $f_2(n) = O(1)$ for all $n$.

Let $\eps>0$ be arbitrary, let $w$ be a sufficiently large quantity depending on $\eps,a,k$, let $\delta>0$ be sufficiently small depending on $w,\eps,a,k$, and assume that $x$ is sufficiently large depending on $\delta, w, \eps, a,k$.  It will suffice to show that
$$ {\mathbb E}_{x/2 \leq p \leq x} |f_2(ap)| \ll \eps.$$
Dividing into residue classes modulo $W \coloneqq \prod_{p \leq w} p$, it suffices to show that
$$ {\mathbb E}_{x/2 \leq p \leq x: p = b\ (W)} |f_2(ap)| \ll \eps$$
for every $b$ coprime to $W$.
In terms of the von Mangoldt function, it will suffice to show that
$$ {\mathbb E}_{x/2 \leq n \leq x: n = b\ (W)} |f_2(an)| \frac{\phi(W)}{W} \Lambda(n) \ll \eps,$$
or equivalently that
$$ {\mathbb E}_{x/2W \leq n \leq x/W} |f_2(a(Wn+b))| \Lambda_{b,W}(n) \ll \eps,$$
where $\Lambda_{b,W}(n) \coloneqq \frac{\phi(W)}{W} \Lambda(Wn+b)$.  We can write the left-hand side as
$$ {\mathbb E}_{x/2W \leq n \leq x/W} f_2(a(Wn+b)) g(n) \Lambda_{b,W}(n) $$
for some sequence $g: \Z \to \mathbb{D}$.  Using the dense model theorem \cite[Proposition 10.3]{gt-linear} (see also \cite[Proposition 8.1]{gt-primes}, \cite[Theorem 7.1]{tz}, \cite[Theorem 4.8]{gow}, \cite[Theorem 1.1]{rttv}, \cite[Theorem 5.1]{cfz}) to the real and imaginary parts of $g\Lambda_{b,W}$, together with the pseudorandom\footnote{Strictly speaking, in \cite{gt-primes} it was assumed that $w$ was a function of $x$ that went to infinity as $x \to \infty$, but this hypothesis is compatible with our selection of parameters, since we assume $x$ to be sufficiently large depending on $w$.} majorant $\nu$ constructed in \cite[Proposition 9.1]{gt-primes} (see also \cite[Proposition 8.1]{cfz}), we can split
$$ g \Lambda_{b,W}(n) = g_1(n) + g_2(n) + g_3(n) $$
for $x/2W \leq n \leq x/W$,
where $g_1$ is bounded pointwise by $O(1)$, $g_2$ obeys the $\ell^1$ smallness bound
\begin{equation}\label{g2}
{\mathbb E}_{x/2W \leq n \leq x/W} |g_2(n)| \ll \eps,
\end{equation}
and $g_3(n)$ is bounded by $O(\nu(n)+1)$ (so in particular ${\mathbb E}_{x/2W \leq n \leq x/W} |g_3(n)| \ll 1$) and obeys the Gowers uniformity bound
\begin{equation}\label{g3-unif}
 {\mathbb E}_{-x/W \leq n,h_1,\dots,h_{k+1} \leq x/W} \prod_{\omega \in \{0,1\}^{k+1}} g_3(n+\omega_1 h_1 + \dots + \omega_k h_k) \ll \delta
\end{equation}
where we extend $g_3$ by zero to the integers. From Theorem \ref{decomp-erg} we have
$$ {\mathbb E}_{x/2W \leq n \leq x/W} f_2(a(Wn+b)) g_1(n) \ll aW\mathbb{E}_{\frac{ax}{2}\leq n\leq 2ax}|f_2(n)|\ll\eps$$
if $x$ is large enough.  From \eqref{g2} we have
$$ {\mathbb E}_{x/2W \leq n \leq x/W} f_2(a(Wn+b)) g_2(n) \ll \eps$$
and so it will suffice to show that
$$ {\mathbb E}_{x/2W \leq n \leq x/W} f_2(a(Wn+b)) g_3(n) \ll \eps.$$
The sequence $n \mapsto f_1(a(Wn+b))$ can be approximated to accuracy $\eps$ by a basic nilsequence, whose underlying nilmanifold and smooth automorphic function $F$ does not depend on $W$ or $b$.  
We apply \cite[Proposition 11.2]{gt-linear}, which decomposes the nilsequence into a part with bounded dual Gowers $U^{k+1}$-norm and uniformly bounded error. Since $g_3$ has  Gowers $U^{k+1}$-norm bounded by $\ll \delta^{2^{-k-1}}$, we conclude that
$$ {\mathbb E}_{x/2W \leq n \leq x/W} f_1(a(Wn+b)) g_3(n) \ll \eps$$
if $\delta$ is small enough, so by the triangle inequality it suffices to show that
$$ {\mathbb E}_{x/2W \leq n \leq x/W} \int_X G_0(T^{h_0 a(Wn+b)} x) \dots G_k(T^{h_k a(Wn+b)} x) \ d\mu(x) g_3(n)\ll \eps.$$
Exchanging the places of integration and summation and applying the Cauchy-Schwarz inequality, it is enough to show that
$$ \int_X \left|{\mathbb E}_{x/2W \leq n \leq x/W}\, g_3(n) G_0(T^{h_0 a(Wn+b)} x) \dots G_k(T^{h_k a(Wn+b)} x) \right|^2\ d\mu(x) \ll \eps,$$
and this follows from\footnote{Strictly speaking, the statement of \cite[Lemma 3]{fhk} is only directly applicable when $h_0,\dots,h_k$ are in arithmetic progression, but the proof of that lemma easily extends to cover arbitrary (distinct) choices of $h_0,\dots,h_k$.} \cite[Lemma 3]{fhk} and \eqref{g3-unif} if $\delta$ is small enough.
\end{proof}

Combining this theorem with Corollary \ref{fcor}, we conclude

\begin{corollary}\label{decomp}  Let the notation and hypotheses be as in Theorem \ref{main}.  Then we may decompose $f = f_1 + f_2$, where $f_1: \Z \to \C$ is a nilsequence of degree $\leq D$, and $f_2: \Z \to \C$ is a sequence such that
$$ \lim_{x \to \infty} {\mathbb E}_{p \leq x} |f_2(ap)| = 0$$
for any fixed nonzero integer $a$.
\end{corollary}

\section{Nilcharacters and their symbols}\label{sec: symb}

In order to use Corollary \ref{decomp}, it will be convenient to (approximately) decompose nilsequences into linear combinations of a special type of basic nilsequence known as a \emph{nilcharacter}. This concept, introduced in \cite{gtz}, generalises the concept of a polynomial phase $n \mapsto e(\alpha_d n^d + \dots + \alpha_1 n + \alpha_0)$.  We also define the notion of a \emph{symbol} of a nilcharacter, which is based on a similar definition\footnote{The notion of symbol in \cite{gtz} was adapted to a hyperfinite interval or box, rather than to the integers, due to the need to perform a ``single-scale'' analysis in that paper rather than the ``asymptotic'' analysis considered here.  Nevertheless, the two notions of symbol are closely analogous.} in \cite{gtz}, and which informally captures the ``top order'' behaviour of a nilcharacter (such as the top coefficient $\alpha_d$ of the above polynomial phase, up to integer or rational shifts).  Unfortunately, due to topological obstructions, we will need to permit nilcharacters to be vector-valued rather than scalar-valued, in order to prevent these nilcharacters from vanishing at one or more points which will cause a number of technical complications (for instance in properly defining the notion of a symbol of a nilcharacter); see \cite[\S 6]{gtz}.

We review the relevant definitions.

\begin{definition}[Nilcharacters]
A \emph{nilcharacter} of degree $d$ is a basic nilsequence\footnote{We will use $\chiup$ to denote nilcharacters, to distinguish slightly from the symbol $\chi$ used in this paper to denote Dirichlet characters.} $\chiup(n) = F(g(n) \Gamma)$ as above, such that $\|F(x)\|_{\C^m}=1$ for all $x \in G/\Gamma$ (using the usual Hilbert space norm on $\C^m$), and such that there exists a continuous homomorphism $\eta: G_d \to \R$ that maps $\Gamma_d$ to the integers, such that
\begin{equation}\label{fgdx}
 F( g_d x ) = e( \eta(g_d) ) F(x) 
\end{equation}
for all $g_d \in G_d$ and $x \in G/\Gamma$.
\end{definition}

\begin{definition}[Symbols]
Two nilcharacters $\chiup: \Z \to \C^m$, $\chiup': \Z \to \C^{m'}$ of degree $d$ are said to be \emph{$d$-equivalent} if the function $\chiup \otimes \overline{\chiup'}: \Z \to \C^{m \times m'}$ is equal to a basic nilsequence of degree $\leq d-1$ (with the convention that the only nilsequences of degree $0$ are the constants).  This is an equivalence relation; see\footnote{Strictly speaking, the results in \cite[Appendix E]{gtz} involve nilcharacters over the nonstandard integers ${}^* \Z$ rather than the standard integers $\Z$, but the arguments carry over without difficulty to the standard setting.} \cite[Lemma E.7]{gtz}.  The equivalence class $[\chiup]_{\mathrm{Symb}^d(\Z)}$ of a degree $d$ nilcharacters up to $d$-equivalence will be called a \emph{symbol} of order $d$, and the space of such equivalence classes will be denoted $\mathrm{Symb}^d(\Z)$.  The operation of tensor product gives rise to an abelian group structure on $\mathrm{Symb}^d(\Z)$, given by the group law
$$[\chiup]_{\mathrm{Symb}^d(\Z)} + [\chiup']_{\mathrm{Symb}^d(\Z)} \coloneqq [\chiup \otimes \chiup']_{\mathrm{Symb}^d(\Z)},$$
negation law
$$-[\chiup]_{\mathrm{Symb}^d(\Z)} \coloneqq [\overline{\chiup}]_{\mathrm{Symb}^d(\Z)},$$
and identity element 
$$0 \coloneqq [1]_{\mathrm{Symb}^d(\Z)};$$
see \cite[Lemma E.8]{gtz} for a verification that this indeed gives the structure of an abelian group.  In particular we have
$$[q \chiup]_{\mathrm{Symb}^d(\Z)} = [\chiup^{\otimes q}]_{\mathrm{Symb}^d(\Z)}$$
for any natural number $q$, where $\chiup^{\otimes q}$ denotes the tensor product of $q$ copies of $\chiup$.
\end{definition}

\begin{example}  If $\alpha_0,\dots,\alpha_d$ are real numbers, then the polynomial phase sequence $n \mapsto e(\alpha_d n^d + \dots + \alpha_0)$ is a nilcharacter of degree $d$; it is equivalent to the identity nilcharacter $1$ if and only if $\alpha_d$ is an integer (if $d=1$) or rational (if $d>1$).  Thus we see that $\mathrm{Symb}^d(\Z)$ contains a copy of $\R/\Z$ (if $d=1$) or $\R/\Q$ (if $d>1$) as a subgroup.  When $d=1$, the polynomial sequence $g(n)$ must take the form $g(n) = g_1^n g_0$ for some $g_1 \in G_1$, and from \eqref{fgdx} we conclude that all degree $1$ nilcharacters take the form $\chiup(n) = c e(\alpha n)$ for some unit vector $c \in \C^m$ and real number $\alpha$.  Two nilcharacters $c e(\alpha n), c' e(\alpha' n)$ are easily seen to be equivalent if and only if $\alpha,\alpha'$ differ by an integer.  As such, we see that $\mathrm{Symb}^1(\Z)$ is in fact isomorphic to $\R/\Z$.  In contrast, in the higher degree case $d>1$ there are many more symbols than the ones coming from polynomial phases.  A near-example of this is given by the ``piecewise smooth degree $2$ nilcharacter'' $n \mapsto e( \{ \alpha n \}  \beta n )$, where $\alpha,\beta$ are real numbers and $\{ \cdot\}$ denotes the fractional part function.  This is not quite a degree $2$ nilcharacter, because the relevant function $F$ (defined on a Heisenberg nilmanifold) is only piecewise smooth, rather than smooth; however by using a suitable partition of unity it can be modified into a genuine degree $2$ nilcharacter (which is now vector-valued rather than scalar-valued), which has a different symbol from any quadratic phase $n \mapsto e(\alpha_2 n^2 + \alpha_1 n + \alpha_0)$ if $\alpha,\beta$ are irrational.  See \cite[\S 6]{gtz} for further discussion of this example and of symbols in general.
\end{example}

For future reference, it will be important to note that symbols behave well under dilations by natural numbers $q \geq 1$, in that
\begin{equation}\label{qad}
 [\chiup(q \cdot)]_{\mathrm{Symb}^d(\Z)} = q^d [\chiup]_{\mathrm{Symb}^d(\Z)};
\end{equation}
see \cite[Lemma E.8(v)]{gtz}.  

The relevance of the symbol notion for us will come through the following equidistribution result:

\begin{proposition}[Equidistribution]\label{equid}  Let $d \geq 1$, and let $\chiup$ be a degree $d$ nilcharacter with non-trivial symbol: $[\chiup]_{\mathrm{Symb}^d(\Z)} \neq 0$.  Then $\lim_{x \to \infty} \E_{n \leq x} \chiup(n) = 0$.
\end{proposition}

\begin{proof}  This is a variant of \cite[Lemma E.11]{gtz}.  For $d=1$, the claim is clear from Fourier analysis, so suppose $d > 1$.  Write $\chiup(n) = F(g(n))$ for a smooth $\Gamma$-automorphic $F: G \to \C^m$ and a polynomial sequence $g: \Z \to G$.  By \cite[Corollary 1.12]{gt}, we may factor
$$ g(n) = \epsilon g'(n) \gamma(n),$$
where $\epsilon \in G$ is constant in $n$, $\gamma$ is rational in the sense that there exists a natural number $r$ such that $\gamma(n)^r \in \Gamma$ for all $n \in \Z$, and $g'$ takes values in a filtered subgroup $G'$ of $G$ (with all subgroups $G'_i$ being connected, simply connected Lie groups), such that $G'_i \cap \Gamma$ is a lattice in $G'_i$ for all $i$, and $g'$ is totally equidistributed in the sense that
$$ \lim_{x \to \infty} \E_{n \leq x} F'(g'(qn+a)) = \int_{G'/\Gamma'} F'\ d\mu$$
whenever $\Gamma'$ is a subgroup of $G'$ commensurate with $G' \cap \Gamma$, $q$ is a natural number, $a$ is an integer, and $F': G' \to \C^m$ is a smooth $\Gamma'$-automorphic function, where $\mu$ denotes the Haar measure on $G'/\Gamma'$.  As $\gamma$ is rational, it is periodic with some period $q$.  On each arithmetic progression $\{ qn+a: n \in \Z\}$, $\gamma(n) \in \gamma(a) \Gamma$, so by the $\Gamma$-automorphic nature of $F$, it suffices to show that
$$ \lim_{x \to \infty} \E_{n \leq x} F( \epsilon g'(qn+a) \gamma(a) ) = 0$$
for each $a$.  

Suppose for contradiction that this fails for some $a$.
The function $g\mapsto  F(\epsilon g \gamma(a))$ is $\Gamma'$-automorphic, where $\Gamma' \coloneqq G' \cap \gamma(a) \Gamma \gamma(a)^{-1}$.  As $\Gamma'$ is commensurate with $G' \cap \Gamma$, we conclude that there is an $a$ for which
$$ \int_{G' / \Gamma'} F( \epsilon g \gamma(a)) d\mu(g) \neq 0.$$
On the other hand, for any $h$ in the central group $G'_d$, we have from translation invariance of the Haar measure that
$$ \int_{G' / \Gamma'} F( \epsilon g \gamma(a)) d\mu(g) = \int_{G' / \Gamma'} F( h \epsilon g \gamma(a)) d\mu(g).$$
Applying \eqref{fgdx}, we conclude that $\eta$ must annihilate the central group $G'_d$.  Quotienting by $G'_d$, we conclude that $n \mapsto F(\epsilon g'(qn+a) \gamma(a))$ is a nilsequence of degree at most $d-1$ for every $a$; since the indicator functions of arithmetic progressions are also nilsequences of degree at most $d-1$ when $d \geq 2$, we conclude that $\chiup$ is equal to a nilsequence of degree at most $d-1$ and thus has vanishing symbol, giving the desired contradiction.
\end{proof}

We have already seen that $\mathrm{Symb}^d(\Z)$ is isomorphic to $\R/\Z$ when $d=1$.
In higher degree $d \geq 2$, the symbol space is more complicated.  However, we do have the following important property.  Call a nilcharacter $\chiup$ of degree $d$ \emph{irrational} if one has $q [\chiup]_{\mathrm{Symb}^d(\Z)} \neq 0$ for all natural numbers $q$.  For instance, if $\alpha$ is a real number, the degree $1$ nilcharacter $n \mapsto e(\alpha n)$ is irrational if and only if $\alpha$ is irrational.  For higher degrees, we have

\begin{lemma} \label{irrat} If $d \geq 2$, then a nilcharacter $\chiup:\mathbb{Z}\to \mathbb{C}^m$ of degree $d$ is irrational if and only if $[\chiup]_{\mathrm{Symb}^d(\Z)} \neq 0$.  In other words,
$\mathrm{Symb}^d(\Z)$ is torsion-free.  
\end{lemma}

In fact, we can modify the arguments of \cite[Appendix E]{gtz} to establish the stronger claim that $\mathrm{Symb}^d(\Z)$ is a vector space over $\Q$, but we will not need to use this fact here.

\begin{proof}  Suppose we have $q [\chiup]_{\mathrm{Symb}^d(\Z)} = 0$ for some natural number $q$ and nilcharacter $\chiup$ of degree $d \geq 2$, then by \eqref{qad} we have $[\chiup(q \cdot)]_{\mathrm{Symb}^d(\Z)} = 0$.  Thus one has a representation
$$ \chiup(qn) =  F(g(n)\Gamma)$$
for all natural numbers $n$ and some basic nilsequence $F(g(n) \Gamma)$ of degree $\leq d-1$.  Since $\chiup$ has magnitude one, we may ensure that $F$ does also.  By (the proof of) \cite[Lemma E.8(vi)]{gtz}, we can write $g(n) = g'(qn)$ for another polynomial sequence $g'$.  Trivially, the $\chiup_1(n):=\chiup(qn)$ satisfies $\mathrm{tr}(\chiup_1(n)\otimes \overline{\chiup_1}(n))=m$, so we conclude that
$$ \mathrm{tr} \E_{n \leq x} \chiup(n) \otimes \overline{F(g'(n) \Gamma)} 1_{q|n} = \frac{m}{q} + o(1)$$
as $x \to \infty$.  By the Cauchy-Schwarz inequality, we have $|\mathrm{tr}(z)|\leq m^{\frac{1}{2}}\|z\|_{\mathbb{C}^{m\times m}}$ for any $z\in \mathbb{C}^{m\times m}$, so by Fourier analysis we have
$$ \limsup_{x \to \infty} \| \E_{n \leq x} \chiup(n) \otimes \overline{F(g'(n) \Gamma)} e(an/q) \|_{\mathbb{C}^{m\times m}} > 0$$
for some integer $a$.  By Proposition \ref{equid} in the contrapositive, we conclude that
$$ [ \chiup \otimes \overline{F(g'(\cdot) \Gamma)} e(a\cdot/q) ]_{\mathrm{Symb}^d(\Z)} = 0;$$
but $\overline{F(g'(\cdot) \Gamma)} e(a\cdot/q) $ is a basic nilsequence of degree $\leq d-1$ and thus has vanishing symbol.  The claim follows.
\end{proof}

Using this lemma, one can decompose nilsequences into irrational nilcharacters plus a periodic sequence:

\begin{proposition}\label{tas}  Let $f:\Z \to \C$ be a degree $\leq d$ nilsequence.  Then $f$ can be expressed as the uniform limit of finite sums of the form
$$ f_0 + \sum_{i=1}^d \sum_{j=1}^{J_i} c_{i,j} \chiup_{i,j}$$
where $f_0: \Z \to \C$ is periodic, $J_1,\dots,J_d$ are non-negative integers, and for each $1 \leq i \leq d$ and $1 \leq j \leq J_i$, $\chiup_{i,j}: \Z \to \C^{m_{i,j}}$ is a degree $i$ irrational nilcharacter, and $c_{i,j}: \C^{m_{i,j}} \to \C$ is a linear functional.
\end{proposition}

\begin{proof}  We may assume inductively that either $d=1$, or $d>1$ and the claim has already been proven for $d-1$.  By a limiting argument, we may assume without loss of generality that $f$ is a basic nilsequence $f(n) = F(g(n))$.  The function $F: G \to \C$ descends to the quotient $G/\Gamma_d$, on which the compact abelian group $G_d/\Gamma_d$ acts.  By Fourier expansion, we may approximate $F$ uniformly by a finite linear combination of smooth functions, each of which obeys \eqref{fgdx} for some character $\eta$; thus without loss of generality we may assume that $F$ obeys \eqref{fgdx}.  We may also rescale so that $|F|<1$ pointwise.  

One can view $G/\Gamma_d$ as a fibre bundle over the quotient nilmanifold $G / G_d$, with fibres isomorphic to $G_d/\Gamma_d$.  By a smooth partition of unity, one can write $1 = \psi_1^2 + \dots + \psi_k^2$ on $G/G_d$, where each $\psi_i: G/G_d \to \R$ is smooth and supported on an open subset $U_i$ of $G/G_d$ that is so small that the fibre bundle becomes trivial; thus, if $\pi: G/\Gamma_d \to G/G_d$ is the quotient map, then $\pi^{-1}(U_i)$ is smoothly isomorphic to $U_i \times (G/G_d)$.  For each $i$, one can use this trivialisation to construct a $\Gamma$-automorphic smooth function $F_i: G \to \C$ obeying \eqref{fgdx} such that $|F_i(x)| = |\psi_i(\pi(x \Gamma))|$ for all $x \in G$; thus $|F_1|^2 + \dots + |F_k|^2 = 1$. If one then sets
$$ \chiup \coloneqq \left(F(g(n)), \sqrt{1-|F(g(n))|^2} F_1(g(n)), \dots, \sqrt{1-|F(g(n))|^2} F_k(g(n))\right),$$
we see that $\chiup$ is a degree $d$ nilcharacter, and $f$ is a linear functional applied to $\chiup$.  If $\chiup$ is irrational, then we are done; if $\chiup$ has vanishing symbol, then we are also done thanks to the induction hypothesis.  By Lemma \ref{irrat} and the identification $\mathrm{Symb}^1(\Z) \equiv \R/\Z$, the only remaining case is when $d=1$ and $\chiup$ is equivalent to $n \mapsto e(an/q)$ for some rational $a/q$, but then $\chiup$ and hence $f$ will be periodic, and we are again done.
\end{proof}

Irrational nilcharacters are ``minor arc'' in the sense that they exhibit cancellation in bilinear sum estimates.  More specifically:

\begin{lemma}  Let $\chiup:\mathbb{Z}\to \mathbb{C}^m$ be an irrational nilcharacter, let $\eps>0$, let $x$ be sufficiently large depending on $\eps$, and let $y$ be sufficiently large depending on $\eps,\chiup,x$.  Then for any bounded sequences $a_n, b_m = O(1)$, one has
$$ \|{\mathbb E}_{p \leq x} {\mathbb E}_{m \leq y} a_p b_m \chiup(pm)\|_{\mathbb{C}^m} \ll \eps.$$
\end{lemma}

\begin{proof}  By expanding out, we see that
\begin{align*}
\left\|\sum_{j\leq J}x_j\right\|_{\mathbb{C}^m}^2= \left\|\sum_{j\leq J}\sum_{j'\leq J}x_j\otimes x_{j'}\right\|_{\mathbb{C}^{m\times m}}.
\end{align*}
Hence it suffices to show that
$$ \|{\mathbb E}_{p, p' \leq x}  a_p a_{p'} {\mathbb E}_{m \leq y} \chiup(pm) \otimes \overline{\chiup}(p'm)\|_{\mathbb{C}^{m\times m}} \ll \eps^2.$$
The diagonal contribution $p = p'$ will be acceptable for $x$ large enough, so by the triangle inequality it suffices to show that
$$ \|{\mathbb E}_{m \leq y} \chiup(pm) \otimes \overline{\chiup}(p'm)\|_{\mathbb{C}^{m\times m}} \ll \eps^2$$
whenever $p \neq p'$ and $y$ is sufficiently large.  But as $\chiup$ is irrational, the symbol of $\chiup(pm) \otimes \overline{\chiup(p'm)}$ is non-trivial by \eqref{qad}, and the claim follows from Proposition \ref{equid}.
\end{proof}

One can transfer this result to the primes (in the $m$ variable):

\begin{lemma}\label{la}  Let $\chiup:\mathbb{Z}\to \mathbb{C}^m$ be an irrational nilcharacter, let $\eps>0$, let $x$ be sufficiently large depending on $\eps$, and let $y$ be sufficiently large depending on $\eps,\chiup,x$.  Then for any bounded sequences $a_n, b_m = O(1)$, one has
$$ \|{\mathbb E}_{p_1 \leq x} {\mathbb E}_{p_2 \leq y} a_{p_1} b_{p_2} \chiup(p_1 p_2)\|_{\mathbb{C}^{m}} \ll \eps$$
where the implied constants can depend on $\chiup$.
\end{lemma}

\begin{proof}  Let $w$ be sufficiently large depending on $\eps$; we may assume $y$ to be sufficiently large depending on $\eps,\chiup,x,w$.  Let $W$ be the product of the primes less than $w$.  It then suffices to show that
$$ \|{\mathbb E}_{p_1 \leq x} {\mathbb E}_{p_2 \leq y: p_2 = c\ (W)} a_{p_1} b_{p_2} \chiup(p_1 p_2)\|_{\mathbb{C}^m} \ll \eps $$
for each $1 \leq c \leq W$ coprime to $W$.  This in turn will follow (for $y$ large enough) from the estimate
$$ \left\|{\mathbb E}_{p \leq x} {\mathbb E}_{m \leq y/W} a_{p} b_{Wm+c} \frac{W}{\phi(W)} \Lambda(Wm+c) \chiup(p (Wm+c))\right\|_{\mathbb{C}^m} \ll \eps.$$
Using the dense model theorem\footnote{As with the proof of Proposition \ref{trap}, one strictly speaking has to view $w$ as a function of $y$ that goes to infinity as $y \to \infty$ in order to apply the dense model theorem, but this is compatible with our choice of parameters.} as in the proof of Proposition \ref{trap}, and assuming $y$ large enough, we can write
$$ b_{Wm+c} \frac{W}{\phi(W)} \Lambda(Wm+c)  = b'_m + b''_m + b'''_{m}$$
where $b'_m = O(1)$ is a bounded sequence, $b''_m$ is a sequence with 
\begin{equation}\label{bpp}
 {\mathbb E}_{m \leq y/W} |b''_m| \ll \eps,
\end{equation}
and $b'''_m$ is a sequence with
\begin{equation}\label{bppp-l1}
 {\mathbb E}_{m \leq y/W} |b'''_m| \ll 1
\end{equation}
and
\begin{equation}\label{bppp-unif}
 {\mathbb E}_{-y/W \leq m,h_1,\dots,h_{k+1} \leq y/W} \prod_{\omega \in \{0,1\}^{k+1}} b'''_{m+\omega_1 h_1 + \dots + \omega_k h_k} \ll \eps^{2^{k+1}},
\end{equation}
where we extend $b_m'''$ by zero to the integers.

From the previous lemma (with $\eps$ replaced by $\eps/W$), we have
$$ \|{\mathbb E}_{p \leq x} {\mathbb E}_{m \leq y/W} a_{p} b'_m \chiup(p (Wm+c))\|_{\mathbb{C}^m} \ll \eps,$$
if $y$ is large enough, while from \eqref{bpp} one has
$$ \|{\mathbb E}_{p \leq x} {\mathbb E}_{m \leq y/W} a_{p} b''_m \chiup(p (Wm+c))\|_{\mathbb{C}^m} \ll \eps$$
and from \eqref{bppp-l1}, \eqref{bppp-unif}, and \cite[Proposition 11.2]{gt-linear} applied to (each component of) the nilsequences $n \mapsto \chiup(p(Wn+c))$ for $p \leq x$ one has
$$ \|{\mathbb E}_{p \leq x} {\mathbb E}_{m \leq y/W} a_{p} b'''_m \chiup(p (Wm+c))\|_{\mathbb{C}^m} \ll \eps$$
with the implied constants depending on $\chiup$.  The claim follows.
\end{proof}

\section{Proof of main theorem}\label{main-sec}

We are now ready to prove Theorem \ref{main}.  Henceforth the notation and assumptions are as in that theorem.

We begin with treating a degenerate case, in which
$$ \sum_p \frac{1-|g_j(p)|}{p} = \infty$$
for some $0 \leq j \leq k$.  Applying Wirsing's theorem \cite{wirsing}, we then have
$$ {\mathbb E}_{n \leq x} |g_j(n)| = o(1) $$
as $x \to \infty$, from which it is easy to conclude from the triangle inequality that $f$ vanishes identically, in which case Theorem \ref{main} is trivially true.  Thus we may assume that
$$ \sum_p \frac{1-|g_j(p)|}{p} < \infty$$
for all $j=0,\dots,k$.  Setting $G \coloneqq g_0 \dots g_k$ as before, we thus have from the triangle inequality that
$$ \sum_p \frac{1-|G(p)|}{p} < \infty.$$
In particular, for any $\eps$, one has $|G(p)| = 1 - O(\eps)$ for all but finitely many $p$.  If $|G(p)| = 1-O(\eps)$, then
$$ |f(a) G(p) - f(ap)| = |f(a) - \overline{G(p)} f(ap)|  + O(\eps)$$
for any $a$.  Applying Theorem \ref{g-isotop}, we conclude that for any natural number $a$ and $\eps>0$, we have
$$ {\mathbb E}_{2^m \leq p < 2^{m+1}} |f(a) - \overline{G(p)} f(ap)| \ll \eps$$
for log-almost all $m$.  In particular, by the triangle inequality we have the approximate reproducing formula
$$ f(a) = {\mathbb E}_{2^m \leq p < 2^{m+1}} \overline{G(p)} f(ap) + O(\eps)$$
for log-almost all $m$.  

Let $\plim_{m \to \infty}$ be a generalised limit functional with the property that $\plim_{m \to \infty} a_m = 0$ whenever $a_m$ is supported on a log-small set; such a generalised limit functional exists by the Hahn-Banach theorem (or the ultrafilter lemma).  (It will be irrelevant to the argument whether this generalised limit functional agrees with the functional that was used to construct $f$.)  Applying this generalised limit functional and then sending $\eps$ to zero, we obtain the exact reproducing formula
\begin{equation}\label{repro}
 f(a) = \plim_{m \to \infty}  {\mathbb E}_{2^m \leq p < 2^{m+1}} \overline{G(p)} f(ap) 
\end{equation}
which we may then iterate\footnote{We are indebted to Maksym Radziwi{\l}{\l} for suggesting this iteration.} to obtain
$$ f(a) = \plim_{m_1 \to \infty} \plim_{m_2 \to \infty}   {\mathbb E}_{2^{m_1} \leq p_1 < 2^{m_1+1}} {\mathbb E}_{2^{m_2} \leq p_2 < 2^{m_2+1}} \overline{G(p_1)}\, \overline{G(p_2)} f(ap_1 p_2).$$
We use Corollary \ref{decomp} to split $f = f_1 + f_2$.  From that corollary, we see that for any $a$ and $\eps>0$ we have
$$ {\mathbb E}_{2^m \leq p < 2^{m+1}} |f_2(ap)| \ll \eps $$
for log-almost all $m$.  By the triangle inequality, we conclude that
\begin{equation}\label{farep}
 f(a) = \plim_{m_1 \to \infty} \plim_{m_2 \to \infty}   {\mathbb E}_{2^{m_1} \leq p_1 < 2^{m_1+1}} {\mathbb E}_{2^{m_2} \leq p_2 < 2^{m_2+1}} \overline{G(p_1)}\, \overline{G(p_2)} f_1(ap_1 p_2) + O(\eps).
\end{equation}
By Proposition \ref{tas}, we may write
$$ f_1(ap) = f_0(ap) + \sum_{i=1}^D \sum_{j=1}^{J_i} c_{i,j} \chiup_{i,j}(ap) + O(\eps)$$
where $f_0$ is periodic, $c_{i,j}$ is a linear transformation, and each $\chiup_{i,j}$ is an irrational nilcharacter of degree $i$.  By Lemma \ref{la}, the contribution of the irrational nicharacters to \eqref{farep} is $O(\eps)$.  We conclude that
$$
 f(a) = \plim_{m_1 \to \infty} \plim_{m_2 \to \infty}   {\mathbb E}_{2^{m_1} \leq p_1 < 2^{m_1+1}} {\mathbb E}_{2^{m_2} \leq p_2 < 2^{m_2+1}} \overline{G(p_1)}\, \overline{G(p_2)} f_0(ap_1 p_2) + O(\eps).
$$
The double generalised limit functional here is periodic in $a$ (with the same period as $f_0$).  This establishes part (i) of Theorem \ref{main}.

To approach part (ii) of Theorem \ref{main}, we begin by establish some asymptotic orthogonality of $f$ with Dirichlet characters.

\begin{proposition}\label{prp}  Suppose that $\chi$ is a Dirichlet character such that $g_0\dots g_k$ does not weakly pretend to be $\chi$.  Then one has ${\mathbb E}_{n \leq x} f(an) \overline{\chi}(n) = o(1)$ as $x \to \infty$ for every natural number $a$.
\end{proposition}

\begin{proof}  Since the function $G(n)=g_0(n)\dots g_k(n)$ does not weakly pretend to be $\chi$, one has
$$ {\mathbb E}_{2^{m} \leq p < 2^{m+1}} \left(1 - \mathrm{Re}(\overline{G(p)} \chi(p))\right) \gg 1$$
for a log-large set of $m$ (this can be seen by taking contrapositives and using the triangle inequality).  In particular, one can select the generalised limit functional $\plim$ used in the above analysis so that
$$ \plim_{m \to \infty} {\mathbb E}_{2^{m} \leq p < 2^{m+1}} \left(1 - \mathrm{Re}(\overline{G(p)} \chi(p))\right) \neq 0$$
or equivalently
\begin{equation}\label{plm}
 \plim_{m \to \infty} {\mathbb E}_{2^{m} \leq p < 2^{m+1}} \overline{G(p)} \chi(p) \neq 1.
\end{equation}

By Theorem \ref{main}(i), $f$ lies within $\eps$ of a periodic function $f_0$.  From \eqref{repro} we conclude that
$$
 f_0(an) = \plim_{m \to \infty}  {\mathbb E}_{2^m \leq p < 2^{m+1}} \overline{G(p)} f_0(apn) + O(\eps)$$
 and thus
$$
 f_0(an) \overline{\chi}(n) = \plim_{m \to \infty}  {\mathbb E}_{2^m \leq p < 2^{m+1}} \overline{G}(p) \chi(p) f_0(apn) \overline{\chi}(pn) + O(\eps)$$
The function $n \mapsto f_0(an) \overline{\chi}(n)$ is periodic and thus has a well-defined mean value $\alpha$. Since $p$ is a large prime, the function $n\mapsto f_0(apn)\overline{\chi}(pn)$ has the same mean value $\alpha$. Taking means (which only requires using finitely many $n$), we conclude that
$$ \alpha = \plim_{m \to \infty}  {\mathbb E}_{2^m \leq p < 2^{m+1}} \overline{G}(p) \chi(p) \alpha + O(\eps)$$
and hence by \eqref{plm} we have
$$ \alpha \ll \eps $$
where the implied constant can depend on $G$ and $\chi$.  Thus the function $n\mapsto f_0(an) \overline{\chi}(n)$ has mean $O(\eps)$, which implies that
$$ {\mathbb E}_{n \leq x} f(an) \overline{\chi}(n) \ll \eps$$
for sufficiently large $x$.  The claim follows.
\end{proof}

We can now prove part (ii) of Theorem \ref{main}.  Let $a$ be an integer, and suppose that $G$ does not weakly pretend to be any Dirichlet character $\chi$.  Let $\eps > 0$.  By Theorem \ref{main}(i), $f$ lies within $\eps$ of a periodic function $f_0$ of some period $q$.  By Proposition \ref{prp}, we have 
$$ \plim_{m \to \infty} {\mathbb E}_{n \leq 2^m} f(an) \overline{\chi}(n) = 0$$
for every Dirichlet character $\chi$, which by Dirichlet character expansion implies that
$$ \plim_{m \to \infty} {\mathbb E}_{n \leq 2^m: n = 1\ (q)} f(an) = 0.$$
Approximating $f$ by $f_0$, we conclude that
$$ \plim_{m \to \infty} {\mathbb E}_{n \leq 2^m: n = 1\ (q)} f_0(an) \ll \eps.$$
But by the periodicity of $f_0$, the left-hand side is $f_0(a)$.  Taking limits, we obtain $f(a)=0$, as required.

A similar argument can be used to prove part (iii) of Theorem \ref{main}.  Suppose that $G$ weakly pretends to be a Dirichlet character $\chi$ of some period $q_0$, then it cannot weakly pretend to be any other Dirichlet character that arises from a different primitive character than $\chi$.  Again, let $\eps>0$, let $a$ be an integer, and let $f_0$ be a periodic function lying within $\eps$ of $f$ of some period $q$; by dilating the period we may assume that $q$ is a multiple of $q_0$.  By Proposition \ref{prp}, we have
$$ \plim_{m \to \infty} {\mathbb E}_{n \leq 2^m} f(an) \overline{\chi}'(n) = 0$$
for every Dirichlet character $\chi'$ that arises from a different primitive character than $\chi$.  By Dirichlet character expansion, we conclude that
$$ \plim_{m \to \infty} {\mathbb E}_{n \leq 2^m: n = b\ (q)} f(an) = \alpha \chi(b)$$
 for any $b$ coprime to $q$ and some $\alpha$ independent of $b$.  The left-hand side is $f_0(ab) + O(\eps)$, thus
$$ f_0(ab) = \alpha \chi(b) + O(\eps)$$
and in particular (replacing $b$ by $1$ and then using the triangle inequality)
$$ f_0(ab) = f_0(a) \chi(b) + O(\eps).$$
If we replace the periodic function $f_0(a)$ by the average $\tilde f_0(a) \coloneqq \E_{1 \leq b \leq q: (b,q)=1} f_0(ab) \overline{\chi}(b)$, then $\tilde f_0$ is still periodic with period $q$, and $\tilde f_0$ stays within $O(\eps)$ of $f_0$ and hence of $f$.  Moreover, for any integer $c$ coprime to $q$, we can use the change of variables $b' = bc\ (q)$ to see that $\tilde f_0$ satisfies the identity
\begin{align*}
\tilde f_0(ac)=\mathbb{E}_{\substack{1\leq b\leq q\\(b,q)=1}}f_0(abc)\overline{\chi}(b)=\mathbb{E}_{\substack{1\leq b'\leq q\\(b',q)=1}}f_0(ab')\overline{\chi}(b')\chi(c)=\chi(c)\tilde f_0(a).
\end{align*}
This gives Theorem \ref{main}(iii).

\section{Applications}\label{applications}

We now prove the corollaries of the main theorem.\\

\begin{proof}[Proof of Corollary \ref{Sig}] Consider first part (i)  of the corollary. With $\eps_0,\dots,\eps_k, A$ as in that corollary, we can write
$$ 1_A(n) = \frac{1}{2^{k+1}} \prod_{j=0}^k (1 + \eps_j \lambda(n+j)).$$
We expand out the product into a main term $\frac{1}{2^{k+1}}$ and error terms $\prod_{j \in J} \lambda(n+j)$ for $J$ a non-empty subset of $\{0,\dots,k\}$.  From Corollary \ref{chkc}, the error terms with $|J| = 1,3$ are negligible when summing over $n$, while from \cite[Theorem 1.3]{tao} the error terms with $|J|=2$ are also negligible.  The claim follows.

We turn to part (ii) of Corollary \ref{Sig}.  Let $J_0 := \{ j=0,1,2,3: \eps_j \neq 0 \}$.  If $|J_0|=4$, then the set $A$ is empty, since for any natural number $n$, at least one of $n,n+1,n+2,n+3$ must be a multiple of $4$ and thus in the zero set of $\mu$.  The claim is therefore clear in this case. Thus we may assume $|J_0| \leq 3$.  We may expand
$$ 1_A(n) = \frac{1}{2^{r}} \prod_{j \in J_0} (\mu^2(n+j) + \eps_j \mu(n+j)) \prod_{j \in \{0,1,2,3\}\setminus J_0} (1 - \mu^2(n+j)).$$
This expands into a main term
$$ \frac{1}{2^{r}} \left(\prod_{j \in J_0} \mu^2(n+j)\right) \prod_{j \in \{0,1,2,3\}\setminus J_0} (1 - \mu^2(n+j)),$$
and into error terms that  involve products of one, two or three factors of $\eps_j \mu(n+j)$. After summing over $n\leq x$, the main term is what we aimed for, and by a classical sieve theoretic computation (see \cite{mirsky}), it is asymptotic to an explicit limit $C(\varepsilon_0,\ldots, \varepsilon_k)$. As with part (i), the error terms that involve one or three factors are negligible by Corollary \ref{chkc}, while the terms involving two factors are negligible by \cite[Theorem 1.3]{tao} (using the expansion $\mu^2(n) = \sum_{d: d^2|n} \mu(d)$).\end{proof}

Next, we establish Corollary \ref{spin}.

\begin{proof}[Proof of Corollary \ref{spin}.] Because the $q_0,\dots,q_k$ are pairwise coprime, any rational of the form $\frac{a_0}{q_0} + \dots + \frac{a_k}{q_k}$ will  be a non-integer if at least one of the $a_j$ is not divisible by $q_j$.  In this case, we see from the prime number theorem in arithmetic progressions that the multiplicative function 
$$n \mapsto e\left( \left(\frac{a_0}{q_0} + \dots + \frac{a_k}{q_k}\right) \Omega(n)\right)$$
cannot weakly pretend to be any Dirichlet character.  Applying Corollary \ref{dec}, we conclude that
$$ \lim_{m \to \infty} \E^{\log}_{x_m/\omega_m \leq n \leq x_m} \prod_{j=0}^k e\left( \frac{a_j}{q_j} \Omega(n+h_j)\right ) = 0$$
whenever at least one of the $a_j$ is not divisible by $q_j$.  Of course, in the remaining case when $q_j|a_j$ for all $j$, the limit here is equal to one. Using the expansion
\begin{align*}
1_{\Omega(n)= \eps \ (q)}=\frac{1}{q}\sum_{a=0}^{q-1}e\left(\frac{a}{q}(\Omega(n)-\eps)\right),
\end{align*}
 we conclude that
$$ \lim_{m \to \infty} \E^{\log}_{x_m/\omega_m \leq n \leq x_m} \prod_{j=0}^k 1_{\Omega(n+h_j) = \eps_j\ (q_j)} = \frac{1}{q_0 \dots q_k}$$
as claimed.\end{proof}

We now give an argument, discovered independently by Will Sawin and by Kaisa Matom\"aki (and very recently also by Klurman and Mangerel \cite[Lemma 5.3]{kl2}), that also gives some bounds on Liouville sign patterns of length four.  The key estimate is the following partial result towards a four-point case of the logarithmically averaged Chowla conjecture:

\begin{proposition}\label{eprop}  Let $1 \leq \omega_m \leq x_m$ be sequences of reals that go to infinity, and let $\plim$ be a generalised limit functional.  Then the quantity
$$ \alpha \coloneqq \plim_{m \to \infty} \E^{\log}_{x_m/\omega_m \leq n \leq x_m} \lambda(n)\lambda(n+1)\lambda(n+2)\lambda(n+3)$$
lies in the interval $[-1/2,1/2]$.
\end{proposition}

\begin{proof}  Shifting $n$ by $n+1$, we also have
$$ \alpha = \plim_{m \to \infty} \E^{\log}_{x_m/\omega_m \leq n \leq x_m} \lambda(n+1)\lambda(n+2)\lambda(n+3) \lambda(n+4).$$
Using the inequalities $ab \geq a+b-1, -a-b-1$ when $a,b \in \{-1,+1\}$ applied to $a \coloneqq \lambda(n)\lambda(n+1)\lambda(n+2)\lambda(n+3)$ and $b \coloneqq \lambda(n+1) \lambda(n+2) \lambda(n+3) \lambda(n+4)$, and noting that $\lambda^2=1$, we conclude that
$$ \plim_{m \to \infty} \E^{\log}_{x_m/\omega_m \leq n \leq x_m} \lambda(n) \lambda(n+4) \geq 2\alpha - 1, -2\alpha -1.$$
But by \cite[Theorem 1.3]{tao}, the left-hand side is zero, giving the claim.
\end{proof}

\begin{corollary} \label{six} For any $\eps = (\eps_0,\eps_1,\eps_2,\eps_3) \in \{-1,+1\}^4$, the set $A_\eps \coloneqq \{ n \in \N: \lambda(n+j) = \eps_j\,\, \forall j=0,\dots,3\}$ has positive lower density.  In particular, all sixteen sign patterns in $\{-1,+1\}^4$ occur infinitely often in the Liouville function.
\end{corollary}

\begin{proof}  If $A_\epsilon$ had zero lower density, then (by arguing as in \cite{mrt-2}) there exist sequences $1 \leq \omega_m \leq x_m$ going to infinity such that
$$ \lim_{m \to \infty} \E^{\log}_{x_m/\omega_m \leq n \leq x_m} 1_{A_\eps}(n) = 0.$$
Expanding $1_{A_\epsilon}$ as in the proof of Corollary \ref{Sig}(i) and using Corollary \ref{dec}, we can write the left-hand side as
$$ \frac{1}{16} (1 + \eps_0 \eps_1 \eps_2 \eps_3 \alpha),$$
where $\alpha$ is the quantity in Proposition \ref{eprop} for some generalised limit functional $\plim$.  This gives the desired contradiction.
\end{proof}

We observe that the proof of the above corollary also gives the bounds
$$ \frac{1}{16}( 1 - \frac{1}{2} ) \leq \plim_{m \to \infty} \E^{\log}_{n \leq m} 1_{A_\eps}(n) \leq \frac{1}{16}( 1 + \frac{1}{2} )$$
for any generalised limit functional $\plim$, and Theorem \ref{chow}(iii) follows.

Lastly, we establish the corollary on equidistribution of additive functions.

\begin{proof}[Proof of Corollary \ref{addcor}] Let $f_0,\dots,f_k, h_1,\dots,h_k$ be as in that corollary.  Let $g_0,\dots,g_k: \N \to S^1$ be the multiplicative functions $g_j(n) \coloneqq e(f_j(n)) = e^{2\pi i f_j(n)}$, which take values in the unit circle $S^1 \coloneqq \{ z \in \C: |z| = 1\}$.  From \eqref{kip} we see that for any integers $a_0,\dots,a_k$, not all zero, we have
$$
\limsup_{x \to \infty} \E^{\log}_{p \leq x} | g_0^{a_0} \dots g_k^{a_k}(p) - 1 | > 0,
$$
which by the Cauchy-Schwarz inequality implies that
$$
\limsup_{x \to \infty} \E^{\log}_{p \leq x} | g_0^{a_0} \dots g_k^{a_k}(p) - 1 |^2 > 0,
$$
or equivalently
$$
\limsup_{x \to \infty} \E^{\log}_{p \leq x}\left(1 - \mathrm{Re}(g_0^{a_0} \dots g_k^{a_k}(p))\right) > 0.
$$
Thus $g_0^{a_0}  \dots g_k^{a_k}$ does not weakly pretend to be $1$.  In fact we may conclude that for any Dirichlet character $\chi$ of some period $q$, $g_0^{a_0} \dots g_k^{a_k}$ does not weakly pretend to be $\chi$, for if it did, $g_0^{\phi(q) a_0} \dots g_k^{\phi(q) a_k}$ would weakly pretend to be the principal character $\chi^{\phi(q)}$, and hence also weakly pretend to be $1$.  Applying Corollary \ref{dec}, we conclude that
$$ \lim_{m \to \infty} \E^{\log}_{x_m/\omega_m \leq n \leq x_m} g_0^{a_0}(n+h_0) \dots g_k^{a_k}(n+h_k) = 0$$
whenever $a_0,\dots,a_k \in \Z$ are not identically zero.  In particular, we see that for any Laurent polynomials\footnote{Laurent polynomials are functions of the form $z\mapsto \sum_{j=-N}^{N}c_jz^j$ with $c_j\in \mathbb{C}$ and $N\in \mathbb{N}$.} $P_0,\dots,P_k: \C \to \C$ with constant coefficients $c_0,\ldots, c_k$, we have
\begin{equation*}
 \lim_{m \to \infty} \E^{\log}_{x_m/\omega_m \leq n \leq x_m} P_0(g_0(n+h_0)) \dots P_k(g_k(n+h_k)) = c_0 \dots c_k.
\end{equation*}
We observe that the space of Laurent polynomials is dense in the space of continuous functions from $S^{1}$ to $\mathbb{C}$ with respect to the sup norm.\footnote{Ordinary polynomials would not be dense in this space, as is seen by considering the conjugation function $z\mapsto \overline{z}$, which is non-analytic.} To see this, we apply the Stone-Weierstrass theorem to the compact Hausdorff space $S^{1}$, and use the fact that Laurent polynomials form an algebra which is closed with respect to conjugation.\\  

Note that for any Laurent polynomial $P$ with constant coefficient $c$ we have 
\begin{align*}
c=\int_{0}^{1} P(e(\theta))\,\, d\theta.    
\end{align*}
In view of this, for any continuous functions $\phi_0,\ldots, \phi_k:S^{1}\to \mathbb{C}$ we have
\begin{equation}\label{limo}
 \lim_{m \to \infty} \E^{\log}_{x_m/\omega_m \leq n \leq x_m} \phi_0(g_0(n+h_0)) \dots \phi_k(g_k(n+h_k)) = \prod_ {j=0}^{k}\int_{0}^{1}\phi_j(e(\theta))\,\, d\theta.
\end{equation}

Let $I_0,\dots,I_k$ be arcs in $\mathbb{R}/\mathbb{Z}$.  For any $\eps>0$, we may find continuous functions $\phi_0,\ldots, \phi_k,\tilde\phi_0,\ldots, \tilde\phi_k:S^{1}\to \mathbb{R}_{\geq 0}$ with 
\begin{align*}
\phi_{j}(e(\theta))\leq 1_{I_j}(\theta),\quad \tilde\phi_{j}(e(\theta))\geq 1_{I_j}(\theta)   
\end{align*}
for all $0\leq j\leq k$ and $\theta \in \mathbb{R}/ \mathbb{Z}$, and also
$$ \int_0^1 \phi_{j}( e(\theta) )\ d\theta \geq |I_j| - \eps,\quad \int_0^1 \tilde\phi_{j}( e(\theta) )\ d\theta \leq |I_j|+ \eps.$$ Hence, we conclude that
$$
 \prod_{j=0}^k(|I_j|-\varepsilon)\leq \limsup_{m \to \infty} \E^{\log}_{x_m/\omega_m \leq n \leq x_m} 1_{I_0}(f_0(n+h_0)) \dots 1_{I_k}(f_k(n+h_k)) \leq \prod_{j=0}^k (|I_j| + \eps).$$
Sending $\varepsilon\to 0$, we obtain the first part of Corollary \ref{addcor}.  The second part then follows by repeating the proof of Corollary \ref{six}.\end{proof}

\appendix

\section{Correlations with multiplicative weights}\label{mult-cor}

In this appendix we show how Theorem \ref{main} implies the following generalisation of itself.

\begin{theorem}[Structure of correlation sequences with multiplicative weights]\label{main-cor}  Let $k \geq 0$, and let $h_0,\dots,h_k$ be integers and $g_0,\dots, g_k:\mathbb{N}\to \mathbb{D}$ any $1$-bounded multiplicative functions.  Let $q_1,\dots,q_k \in \mathbb{N}$.  Let $1\leq \omega_m\leq  x_m$ be real numbers going to infinity, let $\plim$ be a generalised limit, and let $f:\mathbb{N}\to \mathbb{D}$ be the function
$$ f(a) \coloneqq \plim_{m \to \infty} \E^{\log}_{x_m/\omega_m \leq n \leq x_m} g_0(q_1 n+ah_0) \dots g_k(q_k n+ah_k).$$
Then
\begin{itemize}
\item[(i)]  $f$ is the uniform limit of periodic functions $f_i$.
\item[(ii)]  If the product $g_0 \dots g_k$ does not weakly pretend to be $\chi$ for any Dirichlet character $\chi$, then $f$ vanishes identically.  
\item[(iii)]  If instead the product $g_0 \dots g_k$ weakly pretends to be a Dirichlet character $\chi$, then the periodic functions $f_i$ from part (i) can be chosen to be \emph{$\chi$-isotypic} in the sense that one has the identity $f_i(ab) = f_i(a) \chi(b)$ whenever $a$ is an integer and $b$ is an integer coprime to the periods of $f_i$ and $\chi$, as well as to $q_1,\dots,q_k$.
\end{itemize}
\end{theorem}

\begin{proof}  Let $q$ be the least common multiple of the $q_1,\dots,q_k$.  For all $i=0,\dots,k$, we can write
$$ g_i( q_i n + a h_i ) = \tilde g_i( q n + a \tilde h_i )$$
where $\tilde h_i \coloneqq \frac{q}{q_i} h_i$ and $\tilde g_i(n) \coloneqq g_i(\frac{n}{(n,q/q_i)})$.  Note that the $\tilde g_i$ are also $1$-bounded multiplicative functions, and $g_0 \dots g_k$ weakly pretends to be $\chi$ if and only if $\tilde g_0 \dots \tilde g_k$ does.  Thus, by replacing $g_i,q_i,h_i$ with $\tilde g_i, q, \tilde h_i$ if necessary, we can assume that $q_1=\dots=q_k=q$.

Fix $q$; by induction we may assume that the claim has already been proven for all smaller values of $q$.  If $p^j$ divides $q$ and $|g_i(p^j)| < 1$ for some $i,j$, one can express $g_i(p^j)$ as a convex combination of two complex numbers of norm $1$, and hence can express $g_i$ as the convex combination of two $1$-bounded multiplicative functions that agree with $g_i$ at every prime power other than $p^j$.  From this we see that we can assume without loss of generality that $|g_i(p^j)| = 1$ for all prime powers $p^j$ dividing $q$.

As $a$ ranges over the natural numbers, $(a,q)$ ranges over the factors of $q$.  If $q_1 > 1$ divides $q$ and $(a,q)=q_1$, then by the above discussion we have $|g_i(q_1)| = 1$, and we can write
$$ g_i( q n + a h_i ) = g_i(q_1) g_{i,q_1}( \frac{q}{q_1} n + \frac{a}{q_1} h_i )$$
where $g_{i,q_1}$ is the $1$-bounded multiplicative function
$$ g_{i,q_1}(n) \coloneqq \overline{g_i(q_1)} g_i( q_1 n ).$$
Applying the induction hypothesis (with $q$ replaced by $q/q_1$ and $g_i$ by $g_{i,q_1}$), we thus see that the theorem already holds for the function $a \mapsto f(a) 1_{(a,q) = q_1}$.  Thus by linearity, it suffices to establish the claim for the function $a \mapsto f(a) 1_{(a,q)=1}$.
But when $(a,q)=1$, we can write
$$ f(a) = q \plim_{m \to \infty} \E^{\log}_{x_m/\omega_m \leq n \leq x_m} g_0(n+ah_0) \dots g_k(n+ah_k) 1_{n = 0\ (q)},$$
and one can then perform a multiplicative Fourier expansion
$$ 1_{n=0}(q) = \frac{1}{\phi(q)} \sum_{\eta\ (q)} \overline{\eta(a)} \eta(n+a)$$
where $\eta$ ranges over the Dirichlet characters of period $q$.  Applying Theorem \ref{main} (with $k$ replaced by $k+1$, and adding the additional character $\eta$ to the $g_0,\dots,g_k$), we obtain the claim.  
\end{proof}

\end{document}